\documentclass[12pt]{amsart}
\usepackage[margin=1in]{geometry}
\usepackage{amsmath,amsfonts,amsthm,amssymb,bbm}
\usepackage{graphicx,color,dsfont}
\usepackage{enumitem}
\usepackage{fourier}

\newtheorem{theorem}{Theorem}
\newtheorem{lemma}{Lemma}
\newtheorem{proposition}{Proposition}

\newtheorem{corollary}{Corollary}

\newtheorem{claim}{Claim}

 \theoremstyle{definition}
 
 \theoremstyle{remark}

 \numberwithin{equation}{section}

\newcommand{\vertiii}[1]{{\left\vert\kern-0.25ex\left\vert\kern-0.25ex\left\vert #1
    \right\vert\kern-0.25ex\right\vert\kern-0.25ex\right\vert}}

\newcommand{\W}{{\mathcal W}}

\newcommand{\f}[2]{\frac{#1}{#2}}

\newcommand{\cl}{{\mathcal L}}

\newcommand{\al}{\alpha}
\newcommand{\be}{\beta}

\newcommand{\ga}{\gamma}

\newcommand{\de}{\delta}

\newcommand{\la}{\lambda}

\newcommand{\si}{\sigma}

\newcommand{\vp}{\varphi}
\newcommand{\vpt}{\tilde{\vp}}

\newcommand{\rone}{\mathbf R}

\newcommand{\tcl}{\tilde{\cl}}

\newcommand{\dpr}[2]{\langle #1,#2 \rangle}

\newcommand{\cm}{\mathcal M}

\newcommand{\p}{\partial}

\newcommand{\beq}{\begin{equation}}
\newcommand{\eeq}{\end{equation}}
\newcommand{\beqna}{\begin{eqnarray*}}
\newcommand{\eeqna}{\end{eqnarray*}}
\newcommand{\beqn}{\begin{equation*}}
\newcommand{\eeqn}{\end{equation*}}
\newcommand{\bp}{\begin{proof}}
\newcommand{\ep}{\end{proof}}
\newcommand{\bprop}{\begin{proposition}}
\newcommand{\eprop}{\end{proposition}}
\newcommand{\bt}{\begin{theorem}}
\newcommand{\et}{\end{theorem}}
\newcommand{\bex}{\begin{Example}}
\newcommand{\eex}{\end{Example}}
\newcommand{\bc}{\begin{corollary}}
\newcommand{\ec}{\end{corollary}}
\newcommand{\bcl}{\begin{claim}}
\newcommand{\ecl}{\end{claim}}
\newcommand{\bl}{\begin{lemma}}
\newcommand{\el}{\end{lemma}}

\newcommand{\cj}{{\mathcal J}}

\begin{document}

\title[Stable frequency combs in   periodic waveguides]
 {On the generation of  stable Kerr   frequency combs in the  Lugiato-Lefever model of periodic optical waveguides}
%----------Author 1

\thanks{Sevdzhan Hakkaev partially supported by Scientific Grant RD-08-119/2018 of Shumen University.
 Stanislavova is partially supported by  NSF-DMS under grant \# 1516245.   Stefanov    is partially  supported by  NSF-DMS under grant \# 1614734.}

\author[S. Hakkaev]{\sc Sevdzhan Hakkaev}
\address{ Department of Mathematics and Computer Science, Istanbul Aydin University, Istanbul, Turkey}

\email{sevdzhanhakkaev@aydin.edu.tr}

\address{Faculty of Mathematics and Informatics, Shumen University, Shumen, Bulgaria}

\author[M.  Stanislavova]{\sc Milena Stanislavova}
\address{ Department of Mathematics,
University of Kansas,
1460 Jayhawk Boulevard,  Lawrence KS 66045--7523, USA}
\email{stanis@ku.edu}

%----------Author 2
\author[A. Stefanov]{\sc Atanas Stefanov}
\address{ Department of Mathematics,
University of Kansas,
1460 Jayhawk Boulevard,  Lawrence KS 66045--7523, USA}
\email{stefanov@ku.edu}

\subjclass{Primary 35Q55, 35P10; Secondary 42B37, 42B35}

\keywords{Lugiato-Lefever,  Kerr frequency combs, periodic waveguides, stability}

\date{\today}

\begin{abstract}

We consider the Lugiato-Lefever (LL) model of optical fibers. We construct a two parameter family of steady state solutions, i.e. Kerr frequency combs, for small pumping parameter $h>0$ and the correspondingly (and necessarily) small  detuning parameter, $\al >0$. These are $O(1)$ waves, as they are constructed as   bifurcation  from the standard cnoidal  solutions of the  cubic NLS.  We identify the  spectrally  stable ones, and more precisely,  we show that  the spectrum of the linearized operator contains the eigenvalues $0, -2\al$, while the rest of it is  a subset of $ \{\mu: \Re\mu=-\al \}$.  This is in line with the expectations for effectively damped Hamiltonian systems, such as the LL model.

\end{abstract}

\maketitle

\section{Introduction}
High frequency optical combs generated by  micro resonators  is an active area of research, \cite{CY, KHD, LL1, Matsko,MR}. These were experimentally observed, \cite{Holz} and then further studied in the context of concrete micro resonators.
The relevant envelope models were derived from the Maxwell's equation, \cite{LL1} to describe the mechanism of pattern formation in the optical field of a cavity filled with Kerr medium, which is then subjected to a radiation field.  There are numerous papers dealing with the model derivation, as well as further reductions to dimensionless variables, see for example \cite{CM},\cite{LL1}, \cite{MR} among others. The model equation is the following 
\begin{equation}
\label{mar} 
\psi_t+i \be \psi_{xx}+(\ga+i \de) \psi - i |\psi|^2 \psi=F.
\end{equation}
This is the model considered in \cite{DH, DH1} as well as \cite{Ts1, Ts2}. 
We follow slightly different format, which is more popular in the physics literature.  The derivation in \cite{Matsko}, see also \cite{Mat1} in the regime of the so-called whispering gallery mode resonators.   In it, a ``master'' equation is derived.  After non-dimensionalization of the variables, one is  looking at the model\footnote{Clearly the two models are equivalent after multiplication by $i$ and setting some constants to $1$}  
\begin{equation}
\label{1.1}
  iu_t+u_{xx}-u+2|u|^2u=-i\alpha u-h, t\geq 0, -T\leq x\leq T
\end{equation}
where $u$ is the field envelope (complex-valued) function, $t$ is the normalized time, $x$ is the azimuthal coordinate, while $\al>0$ is  the detuning/damping  parameter  and the normalized  pumping  strength parameter is $h>0$.  We are interested in time independent solutions,
{\it 	that is frequency/Kerr combs   $u(t,x)=\vp(x)$ and their stability}, as solutions of the full time dependent problem \eqref{1.1}.  These satisfy the time-independent equation
\begin{equation}
\label{1.15}
-\vp''+\vp - 2|\vp|^2\vp=i\alpha \vp+h ,  -T\leq x\leq T
\end{equation}
A few words about the range of the parameters. Physically, it is advantageous  that the pumping parameter $h$ be small. In fact, the case $h=0$ is used by many authors as a bifurcation point to construct such waves, assuming that one starts with a ``good''  solution at $h=0$. On the other hand, in a recent paper \cite{MR}, the authors have closely  studied the relationship $\al, h$, which supports the existence of Kerr combs. The case $\al=0$ offers another useful starting point for bifurcation analysis, even when $h\neq 0$. This point of view was explored in \cite{Matsko}, by using the results of \cite{BS}, in the related context of the forced NLS model. In fact, one can write explicit solutions of \eqref{1.1} on the whole line (i.e. $T+\infty$), when $\al=0$, \cite{BS}.  Similar construction was carried out in the periodic case in  \cite{QDM}, since one can write an explicit solution in the case $\al=0$ in terms of Jacobi elliptic functions.

In the periodic setting, there are numerous recent developments as to the existence (and subsequently stability) of periodic in $x$ solitary waves, which are closed to constants. In \cite{QDM, MOT, DH, DH1}, the authors have studied stationary solutions of \eqref{1.15}, close to constants.More specifically,  In \cite{DH, DH1}, the authors have studied ``close to constant'' periodic solutions, by looking at bifurcations close to points of Turing instabilities.  In \cite{MOT} the authors have considered the question for asymptotic stability of ``close to  constant''  solutions , given their spectral stability.

Our goal in this paper is to explore the existence and the stability properties of the solutions of \eqref{1.15}, in the physically relevant regime $0<h<<1$.   Mathematically (and also from a physical perspective), it turns out that $\al$ is also necessarily a  small parameter, in fact $\al\sim h$. In addition, we are looking for (large) solutions close to the standard cnoildal solutions for NLS, with $h=\al=0$, as these are well-known in both the theoretical  context and also easily physically realizable. Most importantly, we are interested in such solutions that are dynamically stable as solutions of \eqref{1.1}. We achieve all of these goals, by first constructing a family of such solutions, as long as the necessary conditions on $\al$, to be established below, are met. Next, we provide an explicit characterizations of their spectral stability, in fact, we provide a fairly explicit description of spectrum of the linearized operator, which should be useful in further studies of its semigroup properties. We postpone these considerations for a  future publication.

\subsection{Construction of the solutions} It will be important to understand the behavior of the solutions of \eqref{1.1} in the case when one of the parameters is zero. This is interesting in itself, but it will also give us important clues as to what is important (and reasonable  to expect) in the physically interesting  case $0<\al, h<<1 $.
Let us first discuss an impossibility result that we alluded to above.
\subsubsection{The case $h=0$, $\al\neq 0$: no  steady state  solutions}
\begin{proposition}($h=0$ does not support stationary solutions)
	\label{obs:10}
The equation
	\begin{equation}
	\label{110}
	\vp''-\vp+2|\vp|^2 \vp +i \al \vp=0,
	\end{equation}
	\underline{does not have}  non-trivial classical solutions $\vp_\al$.
\end{proposition}
\begin{proof}
	Let $\varphi_{\alpha}=\varphi_{\alpha, 1}+i\varphi_{\alpha, 2}$ be  a solution of \eqref{110}. Then, we have
	\begin{equation}
	\label{120}
	\left| \begin{array}{ll} \varphi_{\alpha,1}^{''}-\varphi_{\alpha,1}+2(\varphi_{\alpha,1}^2+\varphi_{\alpha, 2}^2)\varphi_{\alpha, 1}-\alpha\varphi_{\alpha, 2}=0 \\
	\\
	\varphi_{\alpha,2}^{''}-\varphi_{\alpha,2}+2(\varphi_{\alpha,1}^2+\varphi_{\alpha, 2}^2)\varphi_{\alpha, 2}+\alpha\varphi_{\alpha, 1}=0.
	\end{array} \right.
	\end{equation}
	Denoting the second order self-adjoint differential operator $L:=-\p_x^2+1 - 2(\varphi_{\alpha,1}^2+\varphi_{\alpha, 2}^2)$, we see that \eqref{120} is a relationship in the form
	$$
	\left| \begin{array}{l}
	L[\vp_{\al,1}]=-\al \vp_{\al,2} \\
	L[\vp_{\al,2}]=\al \vp_{\al,1}.
	\end{array}
	\right.
	$$
	Thus, applying $L$ to to first equation, we obtain $L^2[\vp_{\al,1}]=-\al^2 \vp_{\al,1}$. By taking a dot product with $\vp_{\al,1}$, we obtain
	$$
	0\leq \|L \vp_{\al,1}\|^2=\dpr{L^2[\vp_{\al,1}]}{\vp_{\al,1}}=-\al^2\|\vp_{\al,1}\|^2<0,
	$$
	which  is a contradiction.
\end{proof}

\subsubsection{The case $h>0$, $\al=0$: description of a one parameter steady state  solutions}
In the other endpoint case, that is $\al=0$, one looks for spatially periodic, time-independent  solutions of \eqref{1.1}, $u=\vp(x)$. For  $\varphi (x)$, we have the equation
\begin{equation}\label{1.2}
  \varphi''-\varphi+2\varphi^3=-h, -T\leq x\leq T
\end{equation}
It turns out that this problem has good explicit cnoidal solutions, which we now describe.
 We integrate  once the equation \eqref{1.2}, to  get
 \begin{equation}\label{1.3}
 \varphi'^2=-\varphi^4+\varphi^2-2h\varphi-c,
 \end{equation}
 where $c$ is a constant of integration. Recall that our interest is in the regime $0<h<<1$.
 {\it We demand that  $\zeta_1<\zeta_2<\zeta_3<\zeta_4$ are four real roots of the polynomial $z^4-z^2+2hz+c$}. Then, we  rewrite the equation \eqref{1.3} in the form
 \begin{equation}\label{1.4}
 \varphi'^2=(\zeta_4-\varphi )(\varphi-\zeta_1)(\varphi-\zeta_2)(\varphi-\zeta_3).
 \end{equation}
 The solution of   (\ref{1.4}) is given by
 \begin{equation}\label{1.5}
 \varphi(x)=\frac{\zeta_4(\zeta_3-\zeta_1)+\zeta_1(\zeta_4-\zeta_3)sn^2(\frac{x}{\sqrt{g}}, \kappa)}{(\zeta_3-\zeta_1)+(\zeta_4-\zeta_3)sn^2(\frac{x}{\sqrt{g}}, \kappa)},
 \end{equation}
where $\kappa^2=\frac{(\zeta_4-\zeta_3)(\zeta_2-\zeta_1)}{(\zeta_4-\zeta_2)(\zeta_3-\zeta_1)}$, $g=\frac{2}{\sqrt{(\zeta_4-\zeta_2)(\zeta_3-\zeta_1)}}$, and
\begin{equation}\label{1.6}
\left| \begin{array}{ll}
\zeta_1+\zeta_2+\zeta_3+\zeta_4=0\\
\zeta_1\zeta_2+\zeta_1\zeta_3+\zeta_1\zeta_4+\zeta_2\zeta_3+\zeta_2\zeta_4+\zeta_3\zeta_4=-1\\
\zeta_1\zeta_2\zeta_3+\zeta_1\zeta_2\zeta_4+\zeta_2\zeta_3\zeta_4+
\zeta_1\zeta_3\zeta_4=-2h\\
\zeta_1\zeta_2\zeta_3\zeta_4=c.
\end{array} \right.
\end{equation}
  These are   the solutions that we shall be interested in. These solutions have been found  in the Lugiato-Lefever context in \cite{BS} and \cite{Matsko}, see also \cite{Mat1}. In the whole line case, the explicit formulas appear for the first time in \cite{BS}.

The current construction gives us a parametrization in terms of $c,h$. We now comment on the range of $c$, for which the condition that   the polynomial $z^4-z^2+2hz+c$ has four different and real roots. At least for $h=0$, this is easy to characterize. Namely, the quartic has four real roots exactly when  $c\in (0,\f{1}{4})$. Then, for $0<h<<1$, we clearly must require that
  $c\in (0, \f{1}{4})$ within an error of $O(h)$.

  For future purposes however, it will be beneficial to  parametrize the waves in term of a
  different set of parameters $m,h$, where $m=\min_{-T\leq x\leq T} \vp(x)$. In fact, $m$ is exactly the root $\vp_3$ above, since the explicit solution $\vp$ varies in the interval $[\vp_3, \vp_4]$, hence $c=-m^4+m^2-2hm$.

  We proceed as follows - set $\vp=m+\psi$ in \eqref{1.3}, whence we require that $\psi\geq 0$ and  we obtain the following equation for $\psi$
  \begin{equation}
  \label{1.33}
  (\psi')^2=\psi[-\psi^3-4m\psi^2+(1-6m^2) \psi+(2m-4m^3-2h)].
  \end{equation}
  In order for such $\psi$ to exist, we clearly need $(2m-4m^3-2h)>0$, that is $m\in (0, \f{1}{\sqrt{2}})$ within $O(h)$. Note that this is consistent with the relations $c\in (0,\f{1}{4})$ and $c=-m^4+m^2$, within $O(h)$. In addition, the polynomial $z\to -z^3-4m z^2+(1-6m^2) z+(2m-4m^3-2h)$ has a positive root -   denote the smallest positive root by $\psi_1$.
  In this case there is unique solution to the equation
  \begin{equation}
  \label{15}
  \psi'=-\sqrt{\psi[-\psi^3-4m\psi^2+(1-6m^2) \psi+(2m-4m^3-2h)]}, -T\leq x\leq T,
  \end{equation}
  which satisfies the following
  \begin{itemize}
  	\item $\psi$ is even, decaying in $[0,T]$ (and  so $\psi'(0)=0$),
  	\item $\psi(0)=\psi_1$, $\psi(T)=0$.
  \end{itemize}
  Now, it is much easier to parametrize the roots $\zeta_1, \ldots, \zeta_4$, which will be useful in the sequel. Take again $h=0$, then the final result will be within $O(h)$. We have the equation
  $$
  z^4-z^2=-c=m^4-m^2.
  $$
  This has solutions $z_{1,2}=\pm m, z_{3,4}=\pm \sqrt{1-m^2}$. By the restriction, $m\in (0, \f{1}{\sqrt{2}})$, we have that $\sqrt{1-m^2}>m$, whence we arrive at the following formulas
  \begin{equation}
  \label{20}
  \zeta_1=-\sqrt{1-m^2}+O(h),\  \zeta_2=-m+O(h),\  \zeta_3=m,\  \zeta_4=\sqrt{1-m^2}+O(h).
  \end{equation}

  These are the solutions $\vp=\vp_h$, defined on $[-T,T]$. One can easily recover the standard cnoidal solution $\vp_0$ -  simply put  $h=0$ in the formula \eqref{20}, which is then subsequently used in \eqref{1.5}.
  \subsection{Main results}
  Our first result is about the instability of $\vp=\vp_h$, as solutions for \eqref{1.1}.
  \begin{proposition}
  	\label{prop:10}
  	Let $\al=0$. Then the waves \eqref{1.5} are spectrally unstable solutions of \eqref{1.1}, with a single real eigenvalue,  for all sufficiently small values of $h: 0<h<<1$.
  \end{proposition}
  We now comment on the necessary conditions on $\al$ that   one needs to impose, so that the profile equation \eqref{1.15} supports $O(1)$ solutions. We have the following
  \begin{proposition}
  	\label{prop:90}
  	Let $0<h<<1$. Assume that \eqref{1.15} has a solution $\vp$, so that $\vp=O(1)$. Then, $\al=O(h)$.
  \end{proposition}
  \begin{proof}
  Take a dot product with $\vp=\vp_1+i \vp_2$ in \eqref{1.15}. Then, since its left-hand side is real, taking imaginary parts results in
  $$
  \al \|\vp\|^2 = h \int_{-T}^T  \vp_2(x) dx\leq h \sqrt{2 T} \|\vp\|
  $$
  This implies that $\al=O(h)$.
  \end{proof}

  Now that we know that $\al=O(h)$, we take the ansatz $\al=\al_0 h$.  Our next result describes the existence of waves for $0<h<<1$, $\al=\al_0  h$,  which only holds for  specific range of $\al_0$.  In order to state the result, we introduce the operators $\cl_\pm$,
  \begin{eqnarray*}
  	\cl_{+,h} &=& -\partial_x^2+1-6\varphi_h^2\\
  	\cl_{-,h} &=& -\partial_x^2+1-2\varphi_h^2
  \end{eqnarray*}
  which will be important for our arguments in the sequel. Also, $L_\pm:=\cl_{\pm,0}$.

  \begin{theorem}
  	\label{102}
  	Let $\al_0$ is such that $ 0<\al_0<\f{\dpr{1}{\vp_0}}{\|\vp_0\|^2}$. Then, there exists $h_0=h_0(\al_0)>0$, so that for every $h: 0<h<h_0$ and $\al:=\al_0 h$, there exists a stationary solution $\vp_{\al}=\vp_{\al,1}+i \vp_{\al,2}$ of
  	\eqref{1.1}  or equivalently  solutions $\vp_{\al,1}(h), \vp_{\al,2}(h)$ of the profile system   \eqref{150}. Moreover, there is the following Taylor expansions formula for the coefficients
  	\begin{eqnarray}
  	\label{500}
  	\vp_{\al,1} &=& (a_0+\f{b_0}{2} h D_2^0+O(h^2))  \vp_0+h \Psi_1^0 +O_{\{\vp_0\}^\perp}(h^2)\\
  	\label{510}
  	\vp_{\al,2} &=&  (b_0-\f{a_0}{2} h D_2^0+O(h^2))  \vp_0+h \Psi_2^0 +O_{\{\vp_0\}^\perp}(h^2)
  	\end{eqnarray}
  	where
  	\begin{eqnarray*}
  		& & a_0=\si_0  \f{\|\vp_0\|^2}{\dpr{1}{\vp_0}}, \ \ \ b_0= \al_0  \f{\|\vp_0\|^2} {\dpr{1}{\vp_0}},
  		\si_0=  \pm \sqrt{ \f{\dpr{1}{\vp_0}^2}{\|\vp_0\|^4}-\al_0^2}; \\
  			& & D_2^0 = 8 \f{\dpr{\vp_0^2 L_+^{-1}[1]}{L_-^{-1}[b_0-\al_0 \vp_0]}}{\dpr{1}{\vp_0}}; \\
  		& &  \Psi_1^0=  a_0^2 L_+^{-1}[1]  + b_0 L_-^{-1}[b_0-\al_0\vp_0];  \\
  		& & \Psi_2^0=  a_0 b_0  L_+^{-1}[1]  - a_0 L_-^{-1}[b_0-\al_0\vp_0].
  	\end{eqnarray*}
  \end{theorem}
  {\bf Remarks:}
  \begin{enumerate}
  	\item Note that by Proposition \ref{prop:h0} below, $Ker[L_-]=span[\vp_0], Ker[L_+]=span[\vp_0']$. Therefore,  the expression  $L_-^{-1}[b_0-\al_0\vp_0]$ is well-defined, since by the definition of $b_0$, we have that $b_0-\al_0\vp_0\perp Ker[L_-]$. Similarly, $1\perp Ker[L_+]$, whence $L_+^{-1}[1]$ is well-defined.
  	\item The theorem applies under the more general ansatz $\al=\al_0 h+O(h^2)$. In fact, since its statement is a first order in $h$, the proof in this more general case, goes without any changes or modifications.
  \end{enumerate}

  We also have the following theorem regarding the stability of these solutions.
  \begin{theorem}
  	\label{105}
  	Let  $h, \al_0, \vp_{\al}(h)$ be as in Theorem \ref{102}. Then, $\vp_{\al}(h)$ is stable \underline{if and only if}
  	$$
  \si_0= - \sqrt{ \f{\dpr{1}{\vp_0}^2}{\|\vp_0\|^4}-\al_0^2}.
  	$$
  	In addition, in the stable case, the spectrum of the full linearized operator $\cj \cl_h$ has two real eigenvalues $0$ and $-2\al$, and the rest of the spectrum is on the vertical line $\{\mu: \Re\mu=-\al \}$.  That is, it
  	admits the description
  	$$
  	\si(\cj \cl_h)\subset  \{0\} \cup\{-2\al\}\cup  \{\mu: \Re\mu=-\al \}.
  	$$
  	In the unstable case, which occurs for $$  \si_0=  \sqrt{ \f{\dpr{1}{\vp_0}^2}{\|\vp_0\|^4}-\al_0^2},$$  there is a single real
  	unstable  eigenvalue in the form $\mu_0 \sqrt{h}+O(h)$, where
  	$$
  	\mu_0=\sqrt{\f{\si_0 \|\vp_0\|^2}{-\dpr{L_{+}^{-1}\vp_0}{\vp_0}}}>0.
  	$$
  		
  \end{theorem}

\section{Preliminaries}

 We  now give  the  basic spectral properties of the linearized operators associated with $\vp_h$.

 \subsection{Spectral properties}

 Before listing these properties,  let us state them in the easier case $h=0$, of which we bifurcate as $h\neq 0$.
 \begin{proposition}
 \label{prop:h0}
 The solution $\vp_0$, (which can be written explicitly by taking $h=0$ in \eqref{20}) satisfies
 \begin{itemize}
 \item $\cl_{-,0}\geq 0$, with $\cl_{-,0}[\vp_0]=0$, $\cl_{-,0}|_{\{\vp_0\}^\perp}>0$
 \item $n(\cl_{+,0})=1$, $\cl_{+,0}[\vp'_0]=0$, $0$ is a simple eigenvalue for $\cl_{+,0}$.
 \end{itemize}
 In addition,   the following two relations hold
$$
\dpr{\cl_{+,0}^{-1} \vp_0}{\vp_0}<0, \ \ \dpr{\cl_{+,0}^{-1}[1]}{\vp_0}=0
$$
 \end{proposition}
 {\bf Remark:} The condition $\dpr{\cl_{+,0}^{-1} \vp_0}{\vp_0}<0$ is equivalent to the stability of the wave $\vp_0$, in the context of the periodic NLS problem \eqref{1.1}, with $h=0$.
 \begin{proof}
 Note that $\vp_0$ satisfies \eqref{1.1} with $h=0$, hence $\cl_{-,0}[\vp_0]=0$. Since $\vp_0>0$, it follows
 by Sturm-Liouville theory that $\cl_{-,0}\geq 0$ and $0$ is the bottom of the spectrum and
 $\cl_{-,0}|_{\{\vp_0\}^\perp}>0$.

 Next, we show the properties of $\cl_{+,0}$. First, observe that the function $\vp_0$ can be realized as a (rescaled)
 minimizer of the following constrained minimization problem
 \begin{equation}
 \label{30}
 \left\{
 \begin{array}{l}
 \int_{-T}^T [(\vp'(x))^2 - \vp^4(x)] dx\to \min \\
 \int_{-T}^T \vp^2(x) dx=1
 \end{array}
 \right.
 \end{equation}
 Indeed, the existence of a minimizer of \eqref{30} is standard, moreover such minimizer is necessarily
  bell-shaped. Its Euler-Lagrange equation is in the form
\begin{equation}
\label{31}
\vp''- 2 \vp^3 = -\si \vp,
\end{equation}
which has an unique (bell-shaped) solution, namely a rescaled version of $\vp_0$.
As a consequence of this, it is standard to show that $n(\cl_{+,0})=1$, since $\dpr{\cl_{+,0} \vp_0}{\vp_0}<0$ and as a consequence of the minimization properties of $\vp_0$,
$\cl_{+,0}|_{\{\vp_0\}^\perp}\geq 0$. Also, by a direct differentiation of the profile equation \eqref{1.1} (with $h=0$), we conclude $\cl_{+,0}[\vp'_0]=0$.

Finally, let us show that $Ker[\cl_{+,0}]=span\{\vp_0'\}$. In order to do that, we will show that the  second independent solution of the equation  $\cl_{+,0}[g]=0$ does not belong to the space $L^2_{per.}[-T,T]$. Normally,
 such a solution $g$ can be written down by the reduction of order   formula as follows
$$
g(x)=\vp'_0(x)\int_{a}^x \f{1}{(\vp'_0(y))^2} dy
$$
The problem is that such a formula blows up whenever the interval of integration contains $0$. So, we use an alternative description of the eigenfunction, due to Rofe-Beketov, (see \cite{Teschl},  Exercise 5.11, p. 154)
$$
g(x)=\vp_0'(x) \int_0^x \f{(2-6\vp_0^2(y))((\vp_0'(y))^2-(\vp_0''(y))^2)}{((\vp_0'(y))^2+(\vp_0''(y))^2))^2} dy -
\f{\vp_0''(x)}{(\vp_0'(x))^2+(\vp_0''(x))^2}
$$
This function is well-defined and satisfies $\cl_{+,0}[g]=0$. In order to show that the eigenvalue at zero is simple, it suffices to prove that $g$ is not $2T$ periodic. Clearly, the second part of the formula in $g$ is $2T$ periodic, so we concentrate on showing that the first piece, $g_1(x)$  is non-periodic. In fact, $g_1(-T)=g_1(T)$, since $\vp_0'(T)=\vp'_0(-T)=0$.

We show that in fact $g_1'(-T)\neq g_1'(T)$. Since, $\vp_0''(-T)=\vp_0''(T)\neq 0$, it suffices to show that
$$
\int_0^T \f{(2-6\vp_0^2(y))((\vp_0'(y))^2-(\vp_0''(y))^2)}{((\vp_0'(y))^2+(\vp_0''(y))^2))^2} dy\neq \int_0^{-T} \f{(2-6\vp_0^2(y))((\vp_0'(y))^2-(\vp_0''(y))^2)}{((\vp_0'(y))^2+(\vp_0''(y))^2))^2} dy.
$$
Since the integrand is even, this is equivalent to
\begin{equation}
\label{posle}
\int_0^T \f{(2-6\vp_0^2(y))((\vp_0'(y))^2-(\vp_0''(y))^2)}{((\vp_0'(y))^2+(\vp_0''(y))^2))^2} dy\neq 0.
\end{equation}
We postpone the verification of \eqref{posle} and the proofs of $\dpr{\cl_{+,0}^{-1} \vp_0}{\vp_0}<0$
and $\dpr{\cl_{+,0}^{-1} \vp_0}{1}=0$ for the appendix.  The computations are somewhat  long and technical, but otherwise standard.
 \end{proof}

 We now continue with our investigation of the behavior of $\cl_{\pm}$, when $0<h<<1$. By a simple differentiation of \eqref{1.2}, we still obtain, even for $h\neq 0$, $\cl_{+,h}[\vp']=0$, so $0$ is still an eigenvalue. This is of course due to the translational invariance, which is preserved even after adding $h$.

 In order to set the stage for our later considerations, it is helpful to observe that given the relations \eqref{20},
 $$
 \vp_h=\vp_0+O(h), \ \cl_{\pm,h}=\cl_{\pm,0}+O_{B(L^2)}(h), \ \la_j(\cl_{\pm,h})=\la_j(\cl_{\pm,0})+O(h).
 $$
where we have used the notations $\la_0(\cl)\leq \la_1(\cl)\leq \ldots $ to enumerate the eigenvalues of a self-adjoint operator $\cl$ bounded from below,  in an increasing order.  In particular, it follows that $\la_0(\cl_{+,h})=\la_0(\cl_{+,0})+O(h)<0$ for small values of $h$, whereas $\la_1(\cl_{+,h})=0$, while $\la_2(\cl_{+,h})=\la_2(\cl_{+,0})+O(h)>0$. Thus, the structure of the spectrum for $\cl_{+,h}$ is the same as $\cl_{+,0}$ as described in Proposition \ref{prop:h0}. In particular, the operator $\cl_{+,h}$ is invertible on the subspace of even functions, since $Ker[\cl_{+,h}]=span[\vp'_h]\subset L^2_{odd.}$.

Our next result concerns the structure of $\cl_{-,h}$, when $h\neq 0$. Note that the modulational invariance is  lost after the addition of $h$, which is why the zero eigenvalue for $\cl_{-,0}$ is expected to move away from zero once we   turn on the $h$ parameter. Also, let us record the formula $\cl_{-,h} \vp_h=h$, which is just a restatement of \eqref{1.2}. We have the following lemma.
 \begin{lemma}
\label{le:10}
There exists $h_0>0$, so that for all $|h|<h_0$, we have the following formulas
\begin{eqnarray}
\label{25}
\vp_h &=& \vp_0 + h \cl_{+,0}^{-1}[1]+O(h^2),  \\
\label{27}
 \la_0(\cl_{-,h}) &=& \f{\int_{0}^{T} \vp_0(x)dx}{\int_{0}^{T} \vp_0^2 (x)dx} h+O(h^2)\\
 \label{28}
 \tilde{\vp}_h &=& \vp_0+  \cl_{-,0}^{-1}\left[4 \vp_0^2  \cl_{+,0}^{-1}[1] + \f{\int_{0}^{T} \vp_0(x)dx}{\int_{0}^{T} \vp_0^2 (x)dx}     \vp_0 \right] h +O(h^2),
\end{eqnarray}
where\footnote{The quantity $\left(4 \vp_0^2  \cl_{+,0}^{-1}[1] + \f{\int_{0}^{T} \vp_0(x)dx}{\int_{0}^{T} \vp_0^2 (x)dx}     \vp_0\right)  \perp \vp_0$, whence $\cl_{-,0}^{-1}$ is well-defined}  $\tilde{\vp}_h:  \cl_{-,h}[\tilde{\vp}_h]= \la_0(\cl_{-,h}) \tilde{\vp}_h$ is the  ground state of
$\cl_{-,h}$.
In particular, $\cl_{-,h}>0$ for $0<h<<1$.
\end{lemma}
 {\bf Remark:} A simple perturbation argument shows that $\la_1(\cl_{-,h})=\la_1(\cl_{-,0})+O(h)$, which is well-separated from zero.

 \begin{proof}
 By differentiating with respect to $h$ the profile equation, we obtain $\cl_{+,h}[\p_h \vp_h]=1$. As a consequence, since we know $Ker[\cl_{+,h}]=span[\vp'_h]$ (and hence $1\perp Ker[\cl_{+,h}]$),
 $$
 \p_h\vp_h=\cl_{+,h}^{-1}[1]+\de \vp_h',
 $$
 for some $\de$. We claim $\de=0$. Indeed, we know that $\vp_h$ is an even function, and so is $\p_h \vp_h$. Clearly $\cl_{+,h}$ (and its inverse on $Ker[\cl_{+,h}]^\perp$)  acts invariantly on the even subspace, so $\cl_{+,h}^{-1}[1]$ is even as well. Thus, the odd piece $\de \vp_h'$ is actually zero, whence  $\de=0$. Thus,
 \begin{equation}
 \label{35}
 \vp_h=\vp_0+h \cl_{+,h}^{-1}[1]+O(h^2).
 \end{equation}
 Next, since $\cl_{-,0}$ has a simple eigenvalue at zero, $L_{-,h}$ has a single eigenvalue close to zero,
 in the form $\la_0(\cl_{-,h})=a h+O(h^2)$. Say the corresponding eigenfunction is in the form $\vp_0+ h z$,
 $z\in H^2[-T,T]$. Thus,
\begin{equation}
\label{40}
\cl_{-,h}[\vp_0+ h z]=a h(\vp_0+ h z).
\end{equation}
 However, by \eqref{35},
\begin{eqnarray*}
\cl_{-,h} &=&-\p_x^2+1-2\vp_h^2=-\p_x^2+1-2 \vp_0^2 - 4 h  \cl_{+,h}^{-1}[1] \vp^0  +O(h^2)= \\
&=&\cl_{-,0} - 4 h \vp_0   \cl_{+,h}^{-1}[1] +O(h^2).
\end{eqnarray*}
By taking the first order in $h$ terms in \eqref{40}, we obtain
$$
 \cl_{-,0} z =4 \vp_0^2  \cl_{+,h}^{-1}[1] +a \vp_0.
$$
 Now, take dot product with $\vp_0$. Note that since $1\perp Ker[\cl_{+,h}]$, we have that \\  $\cl_{+,h}^{-1}[1]=\cl_{+,0}^{-1}[1]+O(h)$. Since $\dpr{\cl_{-,0} z }{\vp_0}= \dpr{ z }{\cl_{-,0}[\vp_0]}=0$,
we obtain the relation
$$
a\|\vp_0\|^2+4 \dpr{\vp_0^3}{\cl_{+,0}^{-1}[1]}=0.
$$
Thus,
$
a=-\f{4}{\|\vp_0\|^2} \dpr{\cl_{+,0}^{-1}[\vp_0^3]}{1}.
$
However,  the profile equation can be rewritten as  $$\cl_{+,0}[\vp_0]=-4 \vp_0^3,$$
whence $\cl_{+,0}^{-1}[\vp_0^3]=-\f{1}{4} \vp_0+\de \vp_0'$ and
$$
a= \f{\int_{-T}^{T} \vp_0(x)dx}{\int_{-T}^{T} \vp_0^2 (x)dx}>0.
$$
Also,
$$
z=\cl_{-,0}^{-1}\left[4 \vp_0^2  \cl_{+,0}^{-1}[1] + \f{\int_{-T}^{T} \vp_0(x)dx}{\int_{-T}^{T} \vp_0^2 (x)dx}  \vp_0 \right]+O(h),
$$
which is \eqref{28}.
 \end{proof}

 \subsection{The linearization about $\vp_h$ and the instability  of $\vp_h$ }
 Let $u(t,x)=\varphi(x)+v(t,x)$, where $v$ is a complex-valued function. Plugging  this in \eqref{1.1} and ignoring the contributions of all terms in the form $O(v^2)$, we obtain
 \begin{eqnarray*}
 & & -v_{2t}+v_{1xx}-v_1+6\varphi^2v_1 = 0 \\
& &      v_{1t}+v_{2xx}-v_2+2\varphi^2v_2=0
 \end{eqnarray*}
 This is clearly in the form
  $$
     \cj \left( \begin{array}{ll} \cl_{+,h} & 0 \\0  & \cl_{-,h} \end{array} \right)\left( \begin{array}{ll} v_{1}\\ v_{2} \end{array} \right)= \left( \begin{array}{ll} v_{1t} \\ v_{2t} \end{array} \right)
     $$
      where $\cj=\left( \begin{array}{ll} 0 & 1 \\ -1&0 \end{array} \right)$. Introduce
      $\cl_h:=\left( \begin{array}{ll} \cl_{+,h} & 0 \\0  & \cl_{-,h} \end{array} \right) $.
Taking the ansatz $ \left( \begin{array}{ll} v_{1}\\ v_{2} \end{array} \right)\to e^{\la t} \left( \begin{array}{ll} v_{1}\\ v_{2} \end{array} \right)$
    \begin{equation}
    \label{24}
    \cj\cl_h \left( \begin{array}{ll} v_{1}\\ v_{2} \end{array} \right)= \la \left( \begin{array}{ll} v_{1}\\ v_{2} \end{array} \right)
    \end{equation}
 Thus, the stability of the wave $\vp_h$  is determined from the eigenvalue problem \eqref{24}.
 Following the usual notions of (spectral) stability, we say that the wave is spectrally stable, if \eqref{24} has no non-trivial solutions (that is $\vec{v} \neq 0$ ) ,
 $(\la, \vec{v}): \vec{v}\in H^2[-T,T]$ with $\Re\la>0$.

 As an immediate consequence of the results of Lemma \ref{le:10}, we can conclude the instability for the eigenvalue problem \eqref{24}. Indeed, we have that $n(\cl_{h})=n(\cl_{+,h})+n(\cl_{-,h})=1$, while $Ker[\cl_{-,h}]=\{0\}$,
 $Ker[\cl_{+,h}]=span[\p_x \vp_h]$.  In addition, since   $\dpr{\cl_{-,h}^{-1} [\p_x \vp^h]}{\p_x \vp^h}>0$, by the positivity of $\cl_{-,h}$ (and hence $\cl_{-,h}^{-1}$), we conclude
 $
 n(D)= n(\dpr{ \cl_{-,h}^{-1}[\p_x \vp_h]}{\p_x \vp_h})=0.
 $
 By the instability index counting theory, we conclude that the eigenvalue problem \eqref{20} has a single real unstable eigenvalue for all small values of $h$. This completes the proof of Proposition \ref{prop:10}.
 \subsection{A precise asymptotic for the unstable eigenvalue}
 For the purposes of the analysis of the full problem (that is with $h\neq 0, \al\neq 0$), we need to compute the unstable eigenvalue of the eigenvalue  problem \eqref{24}, at least to leading order in $h$.

 To this end,  for the spectral analysis of \eqref{24}, we are  looking to find the pair $\la=0, \vec{v}=\left(\begin{array}{ll} 0\\ \vp_0 \end{array} \right)$, which solves \eqref{24} for $h=0$.  In other words, we claim that the instability established in Proposition \ref{prop:10} is due to the bifurcation of the zero eigenvalue\footnote{present at $h=0$, corresponding to modulational invariance, which comes with algebraic multiplicity two} onto the real line - one positive and one negative.
 Multiplying  \eqref{24}by $\cj$ and taking the ansatz $v_2=\vp_h+h z, z\in L^2_{even}[-T,T]$ (and observing that   $\cl_{+,h}$ is invertible on the even subspace), we obtain
 \begin{equation}
\label{60}
\cl_{-,h}[\vp_h+ h z]=-\la^2 \cl_{+,h}^{-1}[\vp_h+h z].
 \end{equation}
 Taking into account $\cl_{-,h}[\vp_h]=h$ and $\cl_{+,h}^{-1}[\vp_h]=\cl_{+,0}^{-1}[\vp_0]+O(h)$,
  we arrive at
  \begin{equation}
  \label{70}
  h(1+\cl_{-,h}[z])=-\la^2[\cl_{+,0}^{-1}[\vp_0]+F(h,z)], \ \ F(h,z)=O(h)+O(z).
  \end{equation}
 It becomes clear that the ansatz for the eigenvalue $\la$ must be in the form $\la=a\sqrt{h}+O(h)$, whence by taking dot product of \eqref{70} with $\vp_h$, and taking only $O(h)$ terms
 $$
 -a^2 \dpr{\cl_{+,0}^{-1} \vp_0} {\vp_0}=\int_{-T}^T  \vp_0(x) dx.
 $$
 Recalling $\dpr{\cl_{+,0}^{-1} \vp_0} {\vp_0}<0$, this yields the formula
 $$
 a=\sqrt{\f{\int_{-T}^T  \vp_0(x) dx}{-\dpr{\cl_{+,0}^{-1} \vp_0} {\vp_0}}}>0.
 $$
 Furthermore, \eqref{70} is solvable for small $h$, by the inverse function theorem. In this way, we have shown the following, more precise and quantitative version of Proposition \ref{prop:10}.
 \begin{proposition}
 \label{prop:11}
 There exists $h_0>0$, so that for all $h: 0<h<h_0$, the eigenvalue problem \eqref{24} has the unstable eigenvalue in the form
 $$
 \la_h=\sqrt{\f{\int_{-T}^T  \vp_0(x) dx}{-\dpr{\cl_{+,0}^{-1} \vp_0} {\vp_0}}} \sqrt{h}+O(h).
 $$
 Beyond that, all the other spectrum is stable. In fact,
 $$
 \si(\cj\cl_{h})\setminus \{\la_h, -\la_h\}\subset i {\mathbf R}.
 $$
 \end{proposition}

 \section{The construction of the waves for for $0<h<<1, 0<\al<<1$}
 We now proceed with the construction of the waves in the regime where both parameters $h, \al$  are turned on.  We henceforth assume $h>0$.

 In addition, we wish to keep the solutions in the even class.
 Let   $\varphi_{\alpha}(x)=\varphi_{\alpha, 1}+i\varphi_{\alpha, 2}$ be  a solution of \eqref{1.1}.
 That is
 \begin{equation}
 \label{150}
 \left| \begin{array}{ll} \varphi_{\alpha,1}^{''}-\varphi_{\alpha,1}+2(\varphi_{\alpha,1}^2+\varphi_{\alpha, 2}^2)\varphi_{\alpha, 1}=\alpha\varphi_{\alpha, 2}-h \\
 \\
 \varphi_{\alpha,2}^{''}-\varphi_{\alpha,2}+2(\varphi_{\alpha,1}^2+\varphi_{\alpha, 2}^2)\varphi_{\alpha, 2}=-\alpha\varphi_{\alpha, 1}
 \end{array} \right.
 \end{equation}
For more symmetric formulation, introduce
$$
\vpt_1:= \varphi_{\alpha,1}+\varphi_{\alpha, 2}, \vpt_2:= \varphi_{\alpha,1}-\varphi_{\alpha, 2}.
$$
We have the equations
 \begin{equation}
 \label{151}
 \left| \begin{array}{ll} \vpt_1^{''}-\vpt_1+(\vpt_1^2+\vpt_2^2)\vpt_1=-\alpha\vpt_2-h \\
 \\
 \vpt_2''-\vpt_2+(\vpt_1^2+\vpt_2^2)\vpt_2=\al \vpt_1 - h.
 \end{array} \right.
 \end{equation}
Introducing the operator $\tcl=-\p_x^2+1-(\vpt_1^2+\vpt_2^2)=-\p_x^2+1-2 (\vp_{\al,1}^2+\vp_{\al,2}^2)$, we can rewrite the previous relations in the form
$\tcl[\vpt_1]=\al\vpt_2+h, \tcl[\vpt_2]=h-\al\vpt_1$. Applying $\cl$ to the first equation, we obtain,  in terms of $\vp_{\al,1}, \vp_{\al,2}$
 \begin{equation}
 \label{160}
 \left| \begin{array}{l}
 (\tcl^2+\al^2) [\vp_{\al,1}]=h\tcl[1], \\
 (\tcl^2+\al^2)[\vp_{\al,2}]=\al h.
 \end{array} \right.
 \end{equation}

 It is now useful to perform some analysis in the regime $h<<1$. If $\tcl$ does not have any eigenvalues close to zero, that is $\tcl^2\geq \de^2, \de=O(1)$, we will have from \eqref{160} that $\vp_{\al,1}=O(h), \vp_{\al,2}=O(h)$, whence we have $\tcl=-\p_x^2+1-O(h)$. In this case, one can show that \eqref{160} has (small) solutions, given approximately by
 \begin{equation}
 \label{162}
 \left| \begin{array}{l}
\vp_{\al,1}=h (-\p_x^2+1)^{-1}[1]+O(h^2) \\
 \vp_{\al,2}= \al h (-\p_x^2+1)^{-2} [1]+O(h^2)
 \end{array} \right.
 \end{equation}
 So, we have shown the following.
 \begin{proposition}(Existence of small solutions)
 	There exists $h_0>0$, so that for all $0<h<h_0, \al>0$ there exists a solution of \eqref{160}, in the form \eqref{162}.
 \end{proposition}
 We will proceed to show that such solutions are nicely behaved with respect to stability. Unfortunately, these solutions are not very useful from a practical point of view, since they are small and it is not easy to use them in signal detection devices etc.
 \subsection{$O(1)$ solutions of \eqref{151} - an informal analysis of the profile equation}

 As we have mentioned above, we shall use $h$ as a small parameter, by taking $\al:=\al_0 h, \al_0=O(1)$.  Next, we assume that \eqref{151} (or equivalently \eqref{160}) has a solution.  In addition, we model  $\tcl$ to be a small perturbation of $\cl_{-,0}$. In particular, it has a small (and simple) eigenvalue close to zero (in order to produce $O(1)$ solutions of \eqref{160}), denoted by $\si_h=\si_0 h+O(h^2)$, with a corresponding eigenfunction $\vp_h=\vp_0+O(h)$. In addition, the next eigenvalue is positive and order $O(1)$.

   By projecting \eqref{160}  onto $\vp_h$ and its complementary subspace  $\{\vp_h\}^\perp$, we arrive at the formula
 \begin{eqnarray}
 \label{165}
 \vp_{\al,1} &=& \f{h \si}{\si^2+\al^2} \f{\dpr{1}{\vp_h}}{\|\vp_h\|^2} \vp_h+q_1, \ \ q_1=O(h), q_1 \in \{\vp_h\}^\perp \\
 \label{167}
  \vp_{\al,2} &=& \f{\al h}{\si^2+\al^2} \f{\dpr{1}{\vp_h}}{\|\vp_h\|^2} \vp_h+q_2, \ \ q_2=O(h^2), q_2 \in \{\vp_h\}^\perp
 \end{eqnarray}
 One can in principle continue with the construction of $\vp_{\al,1}, \vp_{\al,2}$ based on \eqref{165} and
 \eqref{167}, but it becomes hard to keep track of the expansion of $\si_h$ in powers of $h$. Instead, we will pass to the known waves $\vp_0$, since we have a good understanding  of the operator $\cl_{-,0}$, which we denote by $L_{-}$ henceforth\footnote{Also, we introduce $L_+:=\cl_{+,0}$}. More precisely, we postulate the form
 \begin{eqnarray}
 \label{170}
 \vp_{\al,1} &=& (a_0+a_1 h+O(h^2))\vp_0+ h \Psi_1, \ \ \Psi_1\perp \vp_0  \\
 \label{175}
  \vp_{\al,2} &=& (b_0+b_1 h+O(h^2))\vp_0+ h \Psi_2,  \ \ \Psi_2\perp \vp_0
 \end{eqnarray}
 Comparing the expansions \eqref{165} with \eqref{170}(and \eqref{167} with \eqref{175} respectively), we have the formula
 \begin{eqnarray}
 \label{180}
  a_0 &=& \f{\si_0}{\si_0^2+\al_0^2} \f{\dpr{1}{\vp_0}}{\|\vp_0\|^2}  \\
 \label{185}
 b_0 &=& \f{\al_0}{\si_0^2+\al_0^2} \f{\dpr{1}{\vp_0}}{\|\vp_0\|^2}
\end{eqnarray}
 Next, using the form of the operator $\tcl$, we have
 $$
 \tcl=-\p_x^2+1 - 2(\vp_{\al,1}^2+\vp_{\al,2}^2)= -\p_x^2+1 - 2(a_0^2+b_0^2)\vp_0^2+O(h).
 $$
 Since we require that $\tcl$ be a perturbation of $L_-$, we must have $a_0^2+b_0^2=1$. This, together with \eqref{180} and \eqref{185} implies that $\si_0$ is completely determined by $\al_0$ and in fact,
 \begin{equation}
 \label{211}
 \si_0^2+\al_0^2=\f{\dpr{1}{\vp_0}^2}{\|\vp_0\|^4}.
 \end{equation}
 We can rewrite the equation \eqref{151} equivalently as follows
 \begin{equation}
 \label{190}
 (\tcl-i\al)[i (\vp_{\al,1}+\vp_{\al,2})+(\vp_{\al,1}-\vp_{\al,2})]=h(1+i)
 \end{equation}
 Denoting $q:=i(\vp_{\al,1}+\vp_{\al,2})+(\vp_{\al,1}-\vp_{\al,2})$, we see that we can write
 $\tcl=-\p_x^2+1 - |q|^2$. In addition, $q$ has the representation
 \begin{equation}
 \label{195}
 q=(c_0+c_1 h+O(h^2))\vp_0+ h \Psi, \ \ \Psi\perp \vp_0,
 \end{equation}
 where clearly $c_0=i(a_0+b_0)+(a_0-b_0)$ can be expressed in terms of $\al_0$. For example, $|c_0|^2=2(a_0^2+b_0^2)=2$. Compute
 \begin{eqnarray*}
 |q|^2 &=& |c_0|^2 \vp_0^2 +2 h \vp_0^2 \Re[c_0\bar{c_1}]+h\vp_0[c_0\bar{\Psi}+\bar{c_0}\Psi]+O(h^2) \\
 &=& 2 \vp_0^2 + h V_\Psi+O(h^2),
 \end{eqnarray*}
 upon introducing $V_\Psi:=2  \vp_0^2 \Re[c_0\bar{c_1}]+\vp_0[c_0\bar{\Psi}+\bar{c_0}\Psi]$.
 It follows that
 $$
 \tcl=-\p_x^2+1 - |q|^2= L_- -h V_\Psi +O(h^2),
 $$
 whence   \eqref{190} becomes
 \begin{equation}
 \label{200}
 (L_--h (V_\Psi+i \al_0) +O(h^2))((c_0+c_1 h+O(h^2))\vp_0+ h \Psi)=h(1+i).
 \end{equation}
 In order to resolve this equation, we need to go in powers of $h$. The terms with power $h^0$ are clearly absent, due to $L_-[\vp_0]=0$, which is just the profile equation. For the first order in $h$ terms, we have the equation
 \begin{equation}
 \label{220}
 L_-\Psi-(V_\Psi+i \al_0)c_0 \vp_0=1+i.
 \end{equation}
 Taking \eqref{220}, and its complex conjugate,  and in addition the form of $V_\Psi$, $|c_0|^2=2$, we arrive at the system
 \begin{equation}
 \label{210}
 \left(\begin{array}{cc}
 -\p_x^2+1- 4\vp_0^2 & -c_0^2 \vp_0^2 \\
 -\bar{c}_0^2 \vp_0^2 & -\p_x^2+1- 4\vp_0^2
 \end{array}\right) \left(\begin{array}{c}
  \Psi \\ \bar{\Psi}
 \end{array}\right)= \left(\begin{array}{c}
 1+i +i \al_0 c_0 \vp_0 +2 \vp_0^3 c_0   \Re[c_0\bar{c_1}] \\  1-i -i \al_0 \bar{c}_0 \vp_0 + 2 \vp_0^3 \bar{c}_0   \Re[c_0\bar{c_1}]
 \end{array}\right)
 \end{equation}
 Diagonalizing the system leads to  the equations
 \begin{eqnarray}
 \label{215}
& &  L_+[\bar{c}_0\Psi+c_0\bar{\Psi}]= c_0+\bar{c}_0+i(\bar{c}_0-c_0)+8 \vp_0^3  \Re[c_0\bar{c_1}],  \\
 \label{217}
& &  L_-[-\bar{c}_0\Psi+c_0\bar{\Psi}] = c_0-\bar{c}_0-i(c_0+\bar{c}_0)-4 i \al_0 \vp_0.
 \end{eqnarray}
 Note that one  solvability condition for \eqref{217} is exactly $\dpr{c_0-\bar{c}_0-i(c_0+\bar{c}_0)-4 i \al_0 \vp_0}{\vp_0}=0$. Elementary computations show that this is equivalent to  $b_0 \dpr{1}{\vp_0}=\al_0 \|\vp_0\|^2$, which is exactly the relation \eqref{185}and \eqref{211}.

 The other relation, is that since $\Psi\perp \vp_0$, we need to have $\bar{c}_0\Psi+c_0\bar{\Psi}\in \{\vp_0\}^\perp$.
 This  imposes the relation $L_+^{-1}[c_0+\bar{c}_0+i(\bar{c}_0-c_0)+8 \vp_0^3  \Re[c_0\bar{c_1}]] \in \{\vp_0\}^\perp$ or equivalently
 $$
 [c_0+\bar{c}_0+i(\bar{c}_0-c_0)] \dpr{ L_+^{-1}[1]}{\vp_0}+8  \Re[c_0\bar{c_1}]\dpr{L_+^{-1}[\vp_0^3]}{\vp_0}=0.
 $$
 In fact, since $\dpr{ L_+^{-1}[1]}{\vp_0}=0$ and $\dpr{L_+^{-1}[\vp_0^3]}{\vp_0}=-\f{1}{4} \dpr{\vp_0}{\vp_0}\neq 0$, it follows that $\Re[c_0\bar{c_1}]=0$.
   It even looks as if we have one degree of  freedom, since $c_1$ is complex valued (and hence two parameters are involved). In the actual non-linear problem however, we need to involve a higher order solvability condition for \eqref{217}, which will finally yield the right number of relations.

 \subsection{Solutions of \eqref{151} - rigorous construction}
 We now setup the full non-linear problem \eqref{151}, with $0<h<<1$, in the equivalent formulation \eqref{190}. More precisely, armed with the results from our informal analysis, we  set the unknown function \\  $q=\vp_{\al,1}-\vp_{\al,2}+i(\vp_{\al,1}+\vp_{\al,2}) \in L^2_{per.}[-T,T]$ in the form
 $$
 q=(c+d h)\vp_0+ h \Psi, \Psi\perp \vp_0
 $$
 where
 $$
 c=a_0-b_0+i(a_0+b_0),
 $$
  and $a_0, b_0$ are given by \eqref{180}, \eqref{185} and \eqref{211}), in terms of $\al_0$. Note that $|c|^2=2$. Also, in accordance with \eqref{211}, we require $\al_0: 0<\al_0 < \f{\dpr{1}{\vp_0}}{\|\vp_0\|^2}$.

Now that we have set $q$ (and in particular $a_0, b_0$), we are looking for a scalar function $d=d(h)$ and a function $\Psi(h)\in \{\vp_0\}^\perp$, so that \eqref{190} holds. We compute
$$
|q|^2=2\vp_0^2 +h V+h^2[|d|^2\vp_0^2+\vp_0(d \bar{\Psi}+\bar{d} \Psi) +|\Psi|^2],
$$
 where
 $$
 V= 2  \vp_0^2 \Re[c\bar{d}]+2 \vp_0\Re[c \bar{\Psi}]
 $$
  is a real-valued function as
  before. Introduce the {\it real-valued function}
 $$
 G=G(d, \Psi)=|d|^2\vp_0^2+\vp_0(d \bar{\Psi}+\bar{d} \Psi) +|\Psi|^2.
 $$
 We thus have a formula for $\tcl$ as follows
 $$
 \tcl=-\p_x^2+1- |q|^2= -\p_x^2+1-2\vp_0^2 - h V - h^2 G=L_- - h V -h^2 G.
 $$
 Plugging this in \eqref{190}, we obtain the following relation
 \begin{equation}
 \label{300}
 (L_- - h(V+i\al_0)-h^2 G)[(c+d h)\vp_0+h \Psi]=h(1+i).
 \end{equation}
 After some algebraic manipulations,  we obtain
  \begin{equation}
  \label{310}
  L_- \Psi-c(V+i \al_0)\vp_0-(1+i)-h[(V+i\al_0)(d \vp_0 +\Psi)+c \vp_0 G]=h^2 G(d, \Psi)\Psi.
  \end{equation}
 Similar to the derivation of \eqref{210}, we take \eqref{310} and its complex conjugate to obtain the following {\it non-linear} in $h$ system of equations
  \begin{eqnarray*}
  & & \left(\begin{array}{cc}
  -\p_x^2+1- 4\vp_0^2 & -c^2 \vp_0^2 \\
  -\bar{c}^2 \vp_0^2 & -\p_x^2+1- 4\vp_0^2
  \end{array}\right) \left(\begin{array}{c}
  \Psi \\ \bar{\Psi}
  \end{array}\right) =  \left(\begin{array}{c}
  1+i +i \al_0 c \vp_0 +2 \vp_0^3 c   \Re[c\bar{d}] \\  1-i -i \al_0 \bar{c} \vp_0 + 2 \vp_0^3 \bar{c}   \Re[c\bar{d}]
  \end{array}\right)+ \\
  &+& h \left(\begin{array}{c}
  	 (V+i\al_0)(d\vp_0 +\Psi)+c \vp_0 G \\
  	 (V-i\al_0)(\bar{d} \vp_0 +\bar{\Psi})+\bar{c}\vp_0 G
  \end{array}\right)+h^2 \left(\begin{array}{c}
   G(d, \Psi)\Psi \\G(d, \Psi)\bar{\Psi}
\end{array}\right)
  \end{eqnarray*}
 Diagonalizing yields the equivalent equations
  \begin{eqnarray}
  \label{315}
    L_+[\bar{c}\Psi+c\bar{\Psi}] &=&  4a_0 +8 \vp_0^3  \Re[c\bar{d}]+h E_1(h,d,\Psi),     \\
   \label{317}
 L_-[-\bar{c}\Psi+c\bar{\Psi}]  &=&    4i( b_0- \al_0 \vp_0) +
 \\
 \nonumber
 &+& 2 i h[ V\vp_0\Im[c\bar{d}]+V \Im[c\bar{\Psi}] - \al_0 \vp_0\Re[c \bar{d}]-\al_0 \Re[c\bar{\Psi}]]+ h^2 E_2,
  \end{eqnarray}
  where $E_1,E_2$ are smooth functions of the respective arguments.
 This is the system that we need to solve - that is, the goal is to find a neighborhood $(0,h_0)$, so that for every $h\in (0,h_0)$, there is a scalar function $d=d(h)$ and a function $\Psi=\Psi(h)\in \{\vp_0\}^\perp$, so that the pair satisfies the previous two relations.

 To that end, we shall use the implicit function theorem.  It is clear that it is more convenient to introduce two real  variables\footnote{recall that $c$ is already fixed in terms of $\al_0$, so finding $D_1, D_2$ is akin to finding the complex number  $d$} $D_1:=\Re[c\bar{d}], D_2:=\Im[c\bar{d}]$. Clearly, the system requires some solvability conditions. We have already established that with our choice of $c$, we have that $c-\bar{c} -i(c+\bar{c})-4 i \al_0 \vp_0\perp \vp_0$. So, from \eqref{317},  we need to require
 \begin{eqnarray*}
 	0 &=& \dpr{V\vp_0\Im(c\bar{d})+V \Im[c\bar{\Psi}]- \al_0 \vp_0\Re[c \bar{d}]-\al_0 \Re [c\bar{\Psi}]}{\vp_0}+O(h)= \\
 	&=&
 	\dpr{D_2 V\vp_0 - D_1 \al_0 \vp_0 + V \Im[c\bar{\Psi}]}{\vp_0}+O(h)
 \end{eqnarray*}
 In the last identity, we used that    $\Psi\perp \vp_0$, whence  by the reality of $\vp_0$, we have that $\bar{\Psi}\perp \vp_0$ as well (and thus  any linear combination of $\Psi, \bar{\Psi}$ is perpendicular to $\vp_0$).  Thus, we end up requiring
 \begin{equation}
 \label{320}
 \dpr{D_2 V\vp_0 - D_1 \al_0 \vp_0 + V \Im[c\bar{\Psi}]}{\vp_0}+O(h)=0
 \end{equation}
 Since  $\bar{c}\Psi+c\bar{\Psi}\perp \vp_0$  ,  we need to have
$$
 0=4 a_0 \dpr{L_+^{-1}[1]}{\vp_0} +8  D_1 \dpr{L_+^{-1}[\vp_0^3]}{1}+O(h).
 $$
 Recalling that $\dpr{L_+^{-1}[1]}{\vp_0}=0$ and   $L_+[\vp_0]=-4\vp_0^3$,  whence  $L_+^{-1}[\vp_0^3]=-\f{1}{4}\vp_0$ and the previous relation reads
 \begin{equation}
 \label{330}
-2D_1 \dpr{\vp_0}{1} +O(h)=0.
 \end{equation}
 The analysis so far allows us to solve the system\eqref{315}, \eqref{317}, for $h=0$.   Namely,  from \eqref{330}, we infer  that
 \begin{equation}
 \label{340}
 D_1^0=0
 \end{equation}
 The next step is to find $\Psi^0$, from \eqref{315}and \eqref{317}, at $h=0$. Inverting   $L_+$ in \eqref{315} and $L_-$ in \eqref{317} and taking the difference,  and taking into account that $\Re[c\bar{d}]=D_1^0+O(h)=O(h)$, we obtain\footnote{Recall that  $b_0-\al_0 \vp_0\perp \vp_0$, so taking $L_-^{-1}$  is justified. Similarly, with the definition of $D_1^0$ in \eqref{340}, taking $L_+^{-1}$ is justified as well.}
  \begin{equation}
 \label{350}
   \Psi^0 =  \f{4 a_0  L_+^{-1}[1]- 4 i L_-^{-1}[ b_0 - \al_0 \vp_0]}{2\bar{c}}= c a_0 L_+^{-1}[1] - i c L_-^{-1}[ b_0 - \al_0\vp_0].
 \end{equation}
 Note that $\Psi^0\perp \vp_0$ (as it should be), since $L_+^{-1}[1]\perp \vp_0$, and
 $Image[ L_-^{-1}]\perp \vp_0$.

  Finally, we use \eqref{320} to determine $D_2^0$. We obtain the formula
 \begin{equation}
 \label{c:14}
 D_2^0\dpr{V^0 \vp_0}{\vp_0}=-\dpr{\Im[c\bar{\Psi}_0]}{V^0 \vp_0}.
 \end{equation}
  We clearly need to compute $\Re[c\bar{\Psi}_0], \Im[c\bar{\Psi}_0]$. We have from \eqref{350},
  \begin{eqnarray*}
\Re[c\bar{\Psi}_0] &=&    2 a_0  L_+^{-1}[1],      \\
\Im[c\bar{\Psi}_0]&=&  2 L_-^{-1}[b_0-\al_0 \vp_0].
  \end{eqnarray*}
  According to its definition
  $$
  V^0=V(0, \Psi_0)=2  \vp_0^2 D_1^0+\vp_0[c \bar{\Psi^0}+\bar{c}\Psi^0]=2\vp_0\Re[c\bar{\Psi}_0]=
  2 a_0  \vp_0 L_+^{-1}[1].
  $$
  Consequently, since $L_+^{-1}[\vp_0^3]=-\f{1}{4} \vp_0$,
   \begin{equation}
   \label{c:27}
   	 \dpr{V^0 \vp_0}{\vp_0} = 4 a_0  \dpr{\vp_0^3}{L_+^{-1}[1]}= - a_0  \dpr{1}{\vp_0}.
   \end{equation}
   Finally,
     \begin{eqnarray*}
   \dpr{\Im[c\bar{\Psi}_0]}{V^0 \vp_0} = 8 a_0 \dpr{\vp_0^2 L_+^{-1}[1]}{L_-^{-1}[b_0-\al_0 \vp_0]}.
      \end{eqnarray*}
      From \eqref{c:14}, we deduce
      \begin{equation}
      \label{c:17}
       D_2^0=  8 \f{\dpr{\vp_0^2 L_+^{-1}[1]}{L_-^{-1}[b_0-\al_0 \vp_0]}}{\dpr{1}{\vp_0}}
      \end{equation}

 To recapitulate, we have determined, in  \eqref{350}, together with $D_1^0, D_2^0$ as determined above,
 the unique solutions of \eqref{315} and \eqref{317}, when $h=0$.
 We now setup the implicit function argument, which will work in a neighborhood of the solution $h=0$, $\Psi^0$, given by \eqref{350}, and $ D_1^0, D_2^0$.

 First, we set the solvability condition arising in \eqref{317}, namely the scalar function\footnote{Here we use again that $\dpr{ \Im[c\bar{\Psi}]}{\vp_0}=0$}
 \begin{eqnarray*}
& & Q_1(h; \Psi, D_1, D_2) = 2i\dpr{D_2  V(D_1, \Psi)\vp_0 +V \Im[c\bar{\Psi}] - \al_0 \vp_0 D_1 -\al_0 \Im[c\bar{\Psi}]}{\vp_0}+ \\
&+& h \dpr{E_2(h,\Psi, D_1, D_2)}{\vp_0}  = 2i\dpr{D_2  V(D_1, \Psi)\vp_0 + V(D_1, \Psi) \Im[c\bar{\Psi}]  - \al_0 \vp_0D_1 }{\vp_0}+\\
&+& h \dpr{E_2(h,\Psi, D_1, D_2)}{\vp_0},
 \end{eqnarray*}
 where $V(D_1, \Psi)=2 D_1 \vp_0^2 + \vp_0(\bar{c} \Psi+c \bar{\Psi})$. The other function is constructed as follows - apply $L_+^{-1}$ in \eqref{315} and $L_-^{-1}$ in \eqref{317} (once we make  sure that the right hand side is orthogonal to $\vp_0$).  After subtracting and simplifying,
 \begin{eqnarray*}
& & Q_2(h; \Psi, D_1, D_2) =  2 \bar{c} \Psi -\left[ 4 a_0 L_+^{-1}[1] - 2  D_1 \vp_0    +h L_+^{-1}[E_1(h; \Psi, D_1,D_2)]  \right]
+4i L_-^{-1}[b_0 -  \al_0 \vp_0] \\
& & + L_-^{-1}[P_{\{\vp_0\}^\perp}[2 i h(D_2 V(D_1, \Psi) \vp_0 -D_1 \al_0\vp_0+V(D_1, \Psi) \Im[c\bar{\Psi}]-\al_0 \Im[c\bar{\Psi}])+h^2 E_2(h; \Psi, D_1,D_2) ]].
 \end{eqnarray*}
 Note that the projection $P_{\{\vp_0\}^\perp}$ becomes irrelevant, once we impose the condition \\
 $Q_1(h; \Psi, D_1, D_2)=0$! On the other hand, we need it in the definition of $Q_2$ to keep it well-defined, even when $Q_1(h; \Psi, D_1, D_2)=0$ is not enforced. We now consider
 $$
 (Q_1, Q_2)(h; \Psi, D_1, D_2): \rone\times \{\vp_0\}^\perp \times \rone\times \rone \to \rone\times L^2_{per.}[-T,T]
 $$
 and we would like to solve
\begin{equation}
\label{485}
 \left|
 \begin{array}{c}
 Q_1(h; \Psi, D_1, D_2)=0 \\
 Q_2(h; \Psi, D_1, D_2)=0.
 \end{array}
 \right.
\end{equation}
 Note that if one obtains solutions to \eqref{485}, the projection $P_{\{\vp_0\}^\perp}$ becomes irrelevant and the system  $Q_1=Q_2=0$ becomes equivalent to the system \eqref{315}and \eqref{317}.
 Observe that by our earlier considerations, for $h=0$, we have a solution, that is
 $$
 \left|
 \begin{array}{c}
 Q_1(0; \Psi^0, 0, D_2^0)=0 \\
 Q_2(0; \Psi^0, 0, D_2^0)=0,
 \end{array}
 \right.
 $$
 where $Q_2^0$ is given in \eqref{c:17}.
 Our construction of the family $\Psi(h), D_1(h), D_2(h)$ in a neighborhood $(0,h_0)$ will follow, once we can verify that
 $$
 d(Q_1, Q_2)(0; \Psi^0, 0, D_2^0)[\cdot, \cdot, \cdot]:  \{\vp_0\}^\perp \times \rone\times \rone \to \rone\times L^2_{per.}[-T,T]
 $$
 is an isomorphism. That is, for every $\chi\in L^2_{per.}[-T,T]$ and $z\in \rone$, there must
 be unique solution $\psi\in \{\vp_0\}^\perp, d_1 \in \rone, d_2\in \rone$  of the linear system
 \begin{eqnarray*}
 & & dQ_1(0; \Psi^0, 0, D_2^0)[\psi, d_1, d_2]=z \\
 & & dQ_2(0; \Psi^0, 0, D_2^0)[\psi, d_1, d_2]=\chi
 \end{eqnarray*}
 so that the linear mapping $(\chi, z)\to (\psi(\chi, z), d_1(\chi, z), d_2(\chi, z))$ is continuous.

 First, we compute $dQ_2(0; \Psi^0, 0, D_2^0)[\psi, d_1, d_2]= 2\bar{c} \psi + 2 d_1 \vp_0$.
 In order to prepare the calculation for $ dQ_1(0; \Psi^0,0, D_2^0)[\psi, d_1, d_2]$, observe that
 \begin{eqnarray*}
V(D_1, \Psi) &=& 2 D_1 \vp_0^2+2\vp_0 \Re[c\bar{\Psi}] \\
dV(0,\Psi^0)(d_1, \psi)  &=& 2 d_1 \vp_0^2+2\vp_0 \Re[c\bar{\psi}]
 \end{eqnarray*}
 Consequently,
  \begin{eqnarray*}
  	 & & dQ_1(0; \Psi^0,0, D_2^0)[\psi, d_1, d_2] =  2i\dpr{d_2 V^0\vp_0+ D_2^0(2 d_1\vp_0^2+2\vp_0\Re[c\bar{\psi}])}{\vp_0}+ \\
  	 &+& 2i\dpr{(2 d_1 \vp_0^2+2\vp_0\Re[c\bar{\psi}]) \Im[c\bar{\Psi}_0]}{\vp_0} +2i \dpr{V^0 \Im[c\bar{\psi}]}{\vp_0} - 2i \al_0 d_1 \|\vp_0\|^2.
  \end{eqnarray*}
 Now, the equation $\chi=dQ_2(0; \Psi^0, 0, D_2^0)[\psi, d_1, d_2]$, has the form
 \begin{eqnarray*}
% \label{380}
 & &\chi=  dQ_2(0; \Psi^0, 0, D_2^0)[\psi, d_1, d_2]= 2\bar{c} \psi + 2 d_1 \vp_0.
 \end{eqnarray*}
  It clearly has the unique solution
  $$
  d_1=\f{\dpr{\chi}{\vp_0}}{2\|\vp_0\|^2}, \psi=\f{1}{2\bar{c}}(\chi- 2d_1 \vp_0)\in \{\vp_0\}^\perp
  $$
   Plugging the expressions for $d_1$ and $\psi$ in the   equation $dQ_1=z$  produces a linear equation for $d_2$, once we take a dot product with $\vp_0$.  More precisely,
   \begin{eqnarray*}
  d_2 \dpr{V^0\vp_0}{\vp_0} &=&   \f{z}{2i} -\dpr{D_2^0(2 d_1\vp_0^2+2\vp_0\Re[c\bar{\psi}])}{\vp_0} -  \dpr{(2 d_1 \vp_0^2+2\vp_0\Re[c\bar{\psi}]) \Im[c\bar{\Psi}_0]}{\vp_0} +\\
  &+& \al_0 d_1 \|\vp_0\|^2 - \dpr{V^0 \Im[c\bar{\psi}]}{\vp_0}
   \end{eqnarray*}
   which has also unique solution, provided the coefficient in front of it, $\dpr{V^0 \vp_0}{\vp_0}\neq 0$  is non-zero. But we have already verfied that, see \eqref{c:27}. We also see that the solution $d_2$ depends linearly,  through a nice formula on $\chi, z$. It follows that the mapping $(\chi, z)\to (\psi, d_1, d_2)$ is indeed an isomorphism, in the sense specified above.  The implicit function theorem thus applies and we have constructed the solutions.

   It remains to verify the  formulas \eqref{500}and \eqref{510}. First, we have that $c\bar{d}=D_1^0+i D_2^0+O(h)=i D_2^0+O(h)$, whence
\begin{equation}
\label{c:30}
d^0= - \f{i c}{2} D_2^0=\f{a_0+b_0}{2} D_2^0 +i\f{b_0-a_0}{2} D_2^0.
\end{equation}
   Thus, starting with the relation $\vp_{\al,1}-\vp_{\al,2}+i (\vp_{\al,1}-\vp_{\al,2})=q=(c+d h)\vp_0 +h \Psi$, we deduce
    \begin{eqnarray*}
     \vp_{\al,1} &=& (a_0+\f{b_0}{2} h D_2^0)\vp_0 +h(a_0^2 L_+^{-1}[1]+b_0 L_-^{-1}[b_0-\al_0 \vp_0]) +O(h^2)\\
     \vp_{\al,2} &=&(b_0-\f{a_0}{2} h D_2^0)\vp_0 +h(a_0 b_0 L_+^{-1}[1]-a_0 L_-^{-1}[b_0-\al_0 \vp_0]) +O(h^2),
    \end{eqnarray*}
 which is the final claim in Theorem \ref{102}.

 \section{Stability analysis for the waves}
  The linearization of \eqref{1.1} around the solution $\vp_\al$ from Theorem \ref{102} is constructed as follows.  Set $u=\vp_\al+v, \vp_\al=\vp_{\al,1}+i \vp_{\al,2}, v=v_1+i v_2$. After ignoring $O(|v|^2)$ terms (and keeping in mind that $\al=\al_0 h$), we obtain the following system
  $$
  \left(\begin{array}{c}
 -\p_t v_2 \\ \p_t v_1
 \end{array}
  \right)=\left(\begin{array}{cc}
 -\p_x^2+1-(6\vp_{\al,1}^2+2\vp_{\al,2}^2) & -4 \vp_{\al,1} \vp_{\al,2}    \\
 -4 \vp_{\al,1} \vp_{\al,2}  &  -\p_x^2+1-(2\vp_{\al,1}^2 + 6\vp_{\al,2}^2)
  \end{array}
  \right)\left(\begin{array}{c}
  v_1 \\ v_2
  \end{array}
  \right)+\al  \left(\begin{array}{c}
    v_2 \\ -v_1
  \end{array}
  \right)
  $$
  Introduce the self-adjoint operator (with domain $H^2(\rone)\times H^2(\rone)$)
  $$
  \cl_h:=\left(\begin{array}{cc}
  -\p_x^2+1-(6\vp_{\al,1}^2+2\vp_{\al,2}^2) & -4 \vp_{\al,1} \vp_{\al,2}    \\
  -4 \vp_{\al,1} \vp_{\al,2}  &  -\p_x^2+1-(2\vp_{\al,1}^2 + 6\vp_{\al,2}^2)
  \end{array}
  \right).
  $$
  In the eigenvalue ansatz, $v_j(t,\cdot)\to e^{\la t} z_j(\cdot)$, the problem becomes
  \begin{equation}
  \label{540}
 \cj \cl_h
 \left(\begin{array}{c}
   z_1 \\ z_2
   \end{array}
   \right) =  (\la+\al)  \left(\begin{array}{c}
   z_1 \\ z_2
   \end{array}
   \right).
  \end{equation}
Introducing $\mu:=\la+\al$ ,  note that \eqref{540} is a Hamiltonian eigenvalue problem in the form $\cj \cl_h \vec{z}=\mu \vec{z}$,
 enjoying all the symmetries that are afforded by the Hamiltonian structure. Let us record it as
 \begin{equation}
 \label{550}
 \cj \cl_h
 \left(\begin{array}{c}
 z_1 \\ z_2
 \end{array}
 \right) =   \mu   \left(\begin{array}{c}
 z_1 \\ z_2
 \end{array}
 \right).
 \end{equation}
 In addition, $\la=0$ and $z_1=\vp_{\al,1}', z_2=\vp_{\al,2}'$ is an eigenvalue (of algebraic multiplicity two) for \eqref{540}, in accordance with the translational invariance of the system \eqref{1.1}.
 {\it Our task here is a bit unusual in that we need to make a good distinction between \eqref{540} and \eqref{550}}.  More precisely, our goal is to find conditions (or actually characterize) the waves that are stable, or equivalently we need to ensure that the eigenvalue problem \eqref{540}, $\la$, satisfy $\Re \la\leq 0$.  In terms of $\mu$, the stability is equivalent to $\Re\mu\leq \al$.
 Here and below, we use the instability index theory developed for eigenvalue problems in the form \eqref{550}, which among other things counts eigenvalues with positive real parts for \eqref{550}. Let us reiterate again, that the existence of those does not necessarily mean instability for \eqref{540}, unless $\Re\mu>\al=\al_0 h$. In fact, we have already one ``instability'' for \eqref{550}, namely an eigenvalue $\mu=\al$, with e-vector $\left(\begin{array}{c} \vp_{\al,1}' \\ \vp_{\al,2}'\end{array} \right)$.

 To this end, we need to track the evolution of the eigenvalues at zero, as we turn on $h$. Before we get on with this task, we need a few preparatory calculations. We  compute
 $
 \W:= \left(\begin{array}{cc}
 6\vp_{\al,1}^2+2\vp_{\al,2}^2 & 4 \vp_{\al,1} \vp_{\al,2}    \\
 4 \vp_{\al,1} \vp_{\al,2}  &   2\vp_{\al,1}^2 + 6\vp_{\al,2}^2
 \end{array}
 \right)
 $
 in powers of $h$. We have
 \begin{eqnarray*}
 	\W &=&  \vp_0^2 \left(\begin{array}{cc}
 		2+4 a_0^2 	  &   4 a_0 b_0    \\
 		4 a_0 b_0     &   2+4 b_0^2
 	\end{array}
 	\right)+ \\
 	&+& 2 h \vp_0 \left(\begin{array}{cc}
 		6a_0 \Psi_1^0 + 2 b_0\Psi_2^0  +2 a_0 b_0 D_2^0 &  2 a_0 \Psi_2^0+2 b_0  \Psi_1^0+(b_0^2-a_0^2)D_2^0      \\
 		2 a_0 \Psi_2^0+2 b_0  \Psi_1^0   +(b_0^2-a_0^2)D_2^0     &    	2 a_0 \Psi_1^0 + 6 b_0\Psi_2^0 -2 a_0 b_0 D_2^0
 	\end{array}
 	\right)+O(h^2).
 \end{eqnarray*}
 Diagonalizing the matrix $\left(\begin{array}{cc}
 2+4 a_0^2 	  &   4 a_0 b_0    \\
 4 a_0 b_0     &   2+4 b_0^2
 \end{array}
 \right)=S \left(\begin{array}{cc}
 6  	  &   0 \\
 0    &  2
 \end{array}
 \right) S^{-1} $, with
 $$
 S=\left(\begin{array}{cc}
 a_0 	  &   -b_0 \\
 b_0   &  a_0
 \end{array}
 \right),\ \  S^{-1}=\left(\begin{array}{cc}
 a_0  & b_0  \\
 -b_0 &  a_0
 \end{array}
 \right)
 $$
 leads to the representation
 $$
  	\W =      S \left[\vp_0^2 \left(\begin{array}{cc}
 		6  	  &   0 \\
 		0    &  2
 	\end{array}
 	\right)+ 2 h \vp_0 [\Psi_1^0  \left(\begin{array}{cc}
 		6a_0	& -2b_0 \\
 		-2b_0  & 2 a_0
 	\end{array}
 	\right) +\Psi_2^0   \left(\begin{array}{cc}
 		6 b_0 &  2 a_0   \\
 		2a_0   & 2b_0
 	\end{array}
 	\right)-D_2^0\left( \begin{array}{cc}
 		0 & 1 \\
 		1 & 0
 	\end{array} \right)] \right]S^{-1}+O(h^2).
$$

 Upon the introduction of the new variables $\left(\begin{array}{c}
 Z_1 \\ Z_2
 \end{array}\right)= S^{-1}  \left(\begin{array}{c}
 z_1 \\ z_2
 \end{array}\right)$, and since $S^{-1} \cj S=\cj$, we can rewrite the eigenvalue problem \eqref{550} in the form
 \begin{eqnarray}
 \label{570}
 \cj \cm_h \left(\begin{array}{c}
 Z_1 \\ Z_2
 \end{array}\right) =\mu \left(\begin{array}{c}
 Z_1 \\ Z_2
 \end{array}\right),
 \end{eqnarray}
 where
 $$
 \cm_h=    \left(\begin{array}{cc}
 L_+ 	  &   0 \\
 0    &  L_-
 \end{array}
 \right) - 2 h \vp_0 [\Psi_1^0  \left(\begin{array}{cc}
 6a_0	& -2b_0 \\
 -2b_0  & 2 a_0
 \end{array}
 \right) +\Psi_2^0   \left(\begin{array}{cc}
 6 b_0 &  2 a_0   \\
 2a_0   & 2b_0
 \end{array}
 \right)-D_2^0\left( \begin{array}{cc}
 0 & 1 \\
 1 & 0
 \end{array} \right)]+O(h^2).
 $$
 This form of the eigenvalue problem is more suggestive of our approach. For $h=0$, we have two dimensional $Ker[\cm_0]$, spanned by the vectors\footnote{Both vectors have one additional generalized eigenvector, so an algebraic multiplicity four at zero}   $\left(\begin{array}{c}
 \vp_0'\\ 0
 \end{array}\right)$ and $\left(\begin{array}{c}
 0 \\ \vp_0
 \end{array}\right)$.
We need to see what the evolution of the  modulational eigenvalue as $h: 0<h<<1$, i.e. the one corresponding to the eigenvector $\left(\begin{array}{c}
0 \\ \vp_0
\end{array}\right)$. This is because, by index counting theory, the instability can only appear in the even subspace of the problem. Also, we can clearly  consider $\cm_h$ instead of $\cl_h$ as the two operators are similar through the matrix $S$.

    \subsection{Tracking the modulational  eigenvalue for  $\cl_h$ $0<h<<1$}
  We set up the following ansatz for the eigenvalue problem for $\cm_h$,
  \begin{equation}
  \label{b:70}
  \cm_h  \left(\begin{array}{c}
   h p_1  \\ \vp_0 + h p_2
  \end{array}\right)=\si h  \left(\begin{array}{c}
   h p_1  \\ \vp_0+ h p_2
  \end{array}\right).
  \end{equation}
    Using the precise form of $\cm_h$, to the leading order $h$, we have
    \begin{eqnarray*}
& & L_+ p_1 + 2 \vp_0^2 (2 b_0 \Psi_1^0-2 a_0 \Psi_2^0-D_2^0)=0 \\
& & L_- p_2 -4\vp_0^2(a_0\Psi_1^0+b_0 \Psi_2^0)=\si \vp_0
    \end{eqnarray*}
    The first equation is resolvable, since $ 2 \vp_0^2 (2 b_0 \Psi_1^0-2 a_0 \Psi_2^0-D_2^0)$ is even and hence perpendicular to $Ker[L_+]=span[\vp_0']$. The solvability condition for the second one, $\si \vp_0 + 4\vp_0^2(a_0\Psi_1^0+b_0 \Psi_2^0)\perp \vp_0$,  is what yields the formula for $\si_0$, namely
    \begin{equation}
    \label{b:90}
    \si_0=-4\f{\dpr{\vp_0^2(a_0\Psi_1^0+b_0 \Psi_2^0)}{\vp_0}}{\|\vp_0\|^2}= -4 a_0\f{\dpr{L_+^{-1}[1]}{\vp_0^3}}{\|\vp_0\|^2}=a_0\f{\dpr{1}{\vp_0}}{\|\vp_0\|^2}.
    \end{equation}

\subsection{Tracking the modulational eigenvalue for $\cj \cl_h$ $0<h<<1$}

Taking cues from the proof of Proposition \ref{prop:11}, we take the following ansatz for the (former) modulation eigenvalue at zero and the corresponding eigenvector $\left(\begin{array}{c}
0 \\ \vp_0 \end{array}\right)$ - take $\mu=\mu_0 \sqrt{h}$ and
$\left(\begin{array}{c}
Z_1 \\ Z_2
\end{array}\right)=\left(\begin{array}{c}
\sqrt{h} q_1  \\ \vp_0+h q_2
\end{array}\right)$.

Plugging this in \eqref{570}, we obtain, after some elementary algebraic manipulations,
\begin{eqnarray*}
 \left(\begin{array}{cc}
L_+ +O(h)  &   O(h)  \\
O(h)    &  L_-  - 4 h \vp_0 (a_0 \Psi_1^0+b_0 \Psi_2^0)
\end{array}
\right)  \left(\begin{array}{c}
 \sqrt{h} q_1  \\ \vp_0+h q_2
\end{array}\right)=\mu_0 \sqrt{h}
\left(\begin{array}{c} -\vp_0-h q_2 \\
	\sqrt{h} q_1
	\end{array} \right)
\end{eqnarray*}
Resolving the first equation, to its leading order $\sqrt{h}$,  yields the relation
$
L_+ q_1 = -\mu_0 \vp_0,
$
or, since $L_+$ is invertible on $\vp_0$,
\begin{equation}
\label{598}
q_1=-\mu_0 L_+^{-1}[\vp_0]+O(\sqrt{h}).
\end{equation}
In the second equation, the leading order is $h$, whence we get the equation
\begin{equation}
\label{600}
L_- q_2 - 4  \vp_0^2(a_0 \Psi_1^0+b_0 \Psi_2^0) = \mu_0 q_1= -\mu_0^2  L_+^{-1}[\vp_0].
\end{equation}
This equation is solvable, provided we ensure $4  \vp_0^2(a_0 \Psi_1^0+b_0 \Psi_2^0)  - \mu_0^2  L_+^{-1}[\vp_0]\perp \vp_0$, as the operator $L_-$ becomes invertible on it, whence
$$
q_2=L_-^{-1}[4  \vp_0^2(a_0 \Psi_1^0+b_0 \Psi_2^0)  - \mu_0^2  L_+^{-1}[\vp_0]].
$$
Thus, we have located the former modulational invariance eigenvalue - namely it is $\mu_0 \sqrt{h}$,
where $\mu_0$ ensures $4  \vp_0^2(a_0 \Psi_1^0+b_0 \Psi_2^0)  - \mu_0^2  L_+^{-1}[\vp_0]\perp \vp_0$.
Equivalently
\begin{equation}
\label{610}
\mu_0^2 = 4 \f{\dpr{\vp_0^2(a_0 \Psi_1^0+b_0 \Psi_2^0)}{\vp_0}}{\dpr{L_+^{-1}\vp_0}{\vp_0}}.
\end{equation}
It remains to compute this last expression. We have
\begin{eqnarray*}
\mu_0^2 &=& 4\f{\dpr{\vp_0^3}{a_0 \Psi_1^0+b_0 \Psi_2^0}}{\dpr{L_+^{-1}\vp_0}{\vp_0}}=
4 \f{a_0}{ \dpr{L_+^{-1}\vp_0}{\vp_0}} \dpr{\vp_0^3}{L_+^{-1}[1]}=  -a_0  \f{\dpr{\vp_0}{1}}{\dpr{L_+^{-1}\vp_0}{\vp_0}}
\end{eqnarray*}
Since $\dpr{L_+^{-1}\vp_0}{\vp_0}<0$, we have that $sgn(\mu_0^2)=sgn(a_0)$. In other words, if $a_0>0$, we have instability, while for $a_0<0$, we have a marginally stable pairs of eigenvalues
$\pm i [\sqrt  \f{-a_0 \dpr{\vp_0}{1}}{\dpr{L_+^{-1}\vp_0}{\vp_0}} \sqrt{h}+O(h)]$

\subsection{Stable and unstable eigenvalues: putting it together }
Before we proceed with our rigorous arguments, let us discuss our findings so far. In the case
$a_0>0$ (or equivalently $\si_0>0$), we have an unstable eigenvalue in the form $\mu_0 \sqrt{h}+O(h)$, where $\mu_0$ is real and
determined from \eqref{610}. The case where $\mu_0$ is purely imaginary
 is more complicated and it needs extra arguments, based on our earlier computations of the sign of the eigenvalues and the index theory.

 Henceforth,  assume $a_0<0$. Also, it is clear that both the even and odd subspaces are invariant under the action of $\cl_h$ and $\cj \cl_h$, so we will consider them separately.
 \subsubsection{Spectral analysis on the even subspace}  In this case, we have established the emergence,
  from the modulational eigenvalue at $h=0$ (which has algebraic multiplicity two!),  of  a pair of marginally stable eigenvalues
  $i [\pm \mu_0  \sqrt{h}+O(h)]$.
  We now compute the Krein index of  the marginally stable pair  of eigenvalues\footnote{which is of course relevant computation, only if $b_0>a_0$.} $ \pm i \sqrt{\mu_0} \sqrt{h} +O(h)$    For a simple pair of eigenvalues, the Krein index coincides with the sign of the expression
  $\dpr{  \Re \left(\begin{array}{c}
  \sqrt{h} q_1  \\ \vp_0+h q_2
  \end{array}\right)}{ \cm_h \Re \left(\begin{array}{c}
  	\sqrt{h} q_1  \\ \vp_0+h q_2
  	\end{array}\right)}$, see \cite{KKS}, p. 267. To this end,  realizing from \eqref{598} that $q_1$ is purely imaginary to the leading order (in the case under consideration),
  \begin{eqnarray*}
  & &  \dpr{  \Re \left(\begin{array}{c}
   		\sqrt{h} q_1  \\ \vp_0+h q_2
   	\end{array}\right)}{ \cm_h \Re \left(\begin{array}{c}
   	\sqrt{h} q_1  \\ \vp_0+h q_2
   \end{array}\right)}  =   -4h \dpr{\vp_0}{\vp_0^2(a_0 \Psi_1^0+b_0\Psi_2^0)}+O(h^{3/2})= \\
  &=&
-4 a_0 h \dpr{\vp_0^3}{L_+^{-1}[1]}+O(h^{3/2})= a_0 h  \dpr{\vp_0}{1}+O(h^{3/2}).
\end{eqnarray*}
  It follows that for $a_0<0$,  the problem has a pair of eigenvalues $ \pm i \sqrt{\mu_0} \sqrt{h} +O(h)$  with a negative Krein signature.

  Recall that $n(\cl_0)=1$, which arises in the even subspace.  Moreover, $dim(Ker[\cl_0])=2$, but recall that one of them occurs in the even subspace, the other one in the odd subspace. On the other hand\footnote{For a self-adjoint, bounded from below operator $M$, acting invariantly on the even subspace, we denote $n_{even}(M)$ the number of negative eigenvalues, when acting on the even subspace}
  $n_{even}(\cl_h)$ could assume the values of
 one  or two - $n_{even}(\cl_h)=1$,   if the translational eigenvalue \eqref{b:50} is positive or  $n_{even}(\cl_h)=2$, if the   modulational eigenvalue \eqref{b:50} is negative. Indeed, $\cl_{0, even}$ has one negative eigenvalue, according to Proposition \ref{prop:h0} - which remains negative, after the perturbation. In addition, the modulational eigenvalue at zero has moved (for $a_0<0$) to the left, see \eqref{b:90}, creating a second negative eigenvalue.

  We claim that
 $n(\cl_h)=2$, at least for small $0<h<<1$.  Indeed, since by the index counting formulas
 $$
   0\leq n_{unstable, even}(\cj \cl_h)\leq n(\cl_h)-\# \{\la\in i \rone: \la \ \ \textup{has negative Krein signature}  \}\leq
 n_{even}(\cl_h)-2\leq 0.
  $$
since we have already verified that there is at least one  pair of purely imaginary eigenvalues, $\pm i \sqrt{\mu_0} \sqrt{h} +O(h)$  with negative Krein signatures (and even eigenfunctions!). Thus, we obtain equalities everywhere in the inequalities above, in particular $n_{unstable, even }(\cj \cl_h)=0$, \\ $n(\cl_{h, even})=2$.

In other words,      the  modulational  eigenvalue has moved to the left (for $a_0<0$), creating  a total of
two negative eigenvalue for $\cl_h$  - one that existed for $\cl_0$ (and it is still there, after the perturbation), to which another one is added.  These  are subsequently offset by two purely imaginary eigenvalues, with negative Krein signatures.

\subsubsection{Spectral analysis on the odd subspace}

For the odd subspace, the index count is somewhat simpler.
Recall  there is the ``unstable'' eigenvalue at $\mu=\al=\al_0 h$ for the eigenvalue problem \eqref{570}, which
corresponds to the translational invariance,  see \eqref{550} and the discussion immediately after. So, this implies that $n_{unstable, odd}(\cj \cl_h)\geq 1$. On the other hand, $n_{odd}(L_+)\leq 1$, since the translational eigenvalue has moved either to the left (in which case $n_{odd}(L_+)=1$), or to the right, and then\footnote{See the Appendix for a calculation involving the direction of the move, which turns out to be  inconclusive to first order}  $n_{odd}(L_+)=0$. Applying the index counting theory, we have that on the other hand $n_{odd}(L_+)\geq n_{unstable,odd}(\cj \cl_h)$, whence $1=n_{unstable, odd}(\cj \cl_h)=n_{odd}(L_+)$. In particular, the translational eigenvalue has moved to the left to become a negative eigenvalue  for $\cl_{+,h}$. In addition, the ``unstable eigenvalue'' is exactly the one computed explicitly, namely $\mu=\al$ and there are no other unstable eigenvalues in the odd subspace. Note that by the Hamiltonian symmetry the eigenvalues of \eqref{570}are symmetric with respect to the imaginary axes, so there is another, stable  one at $\mu=-\al$. We mention that there could be (and there is!) a lot of (point) spectrum on the imaginary axes as well.

In terms of the  original spectral
variables, by combining the conclusions for the even and odd subspaces, we obtain
$$
\la=\mu-\al\subset \left\{\{\al\}\cup\{-\al\} \cup \{\la: \Re \la=0\} \right\} - \al =   \{0\}\cup\{-2\al\} \cup \{\la: \Re \la=-\al\}.
$$
  This is exactly the statement of Theorem  \ref{105} and the proof is completed.
\appendix
   \section{Tracking the translational eigenvalue for  $\cl_h$ $0<h<<1$}
We set up the following ansatz for the eigenvalue problem for $\cm_h$,
   \begin{equation}
   \label{b:30}
   \cm_h  \left(\begin{array}{c}
   \vp_0'+h p_1  \\ h p_2
   \end{array}\right)=\si h  \left(\begin{array}{c}
   \vp_0'+h p_1  \\ h p_2
   \end{array}\right).
   \end{equation}
   In other words,
   \begin{equation}
   \label{b:40}
   \left(\begin{array}{cc}
   L_+ - 12 h \vp_0 (a_0 \Psi_1^0+b_0 \Psi_2^0) & O(h)  \\
   2h \vp_0(2 b_0\Psi_1^0-2 a_0 \Psi_2^0-D_2^0)& L_-+O(h)
   \end{array}\right)   \left(\begin{array}{c}
   \vp_0'+h p_1  \\ h p_2
   \end{array}\right)=\si h  \left(\begin{array}{c}
   \vp_0'  +h p_1  \\ h p_2
   \end{array}\right).
   \end{equation}
   To the leading order $h$, we have the equations
   \begin{eqnarray*}
   	& & 	L_+ p_1 - 12 \vp_0 (a_0 \Psi_1^0+b_0 \Psi_2^0) \vp_0'=\si \vp_0' \\
   	& & L_- p_2 + 2 \vp_0(2 b_0\Psi_1^0 - 2 a_0 \Psi_2^0-D_2^0) \vp_0'=0
   \end{eqnarray*}
   The second equation  always has a solution as $2 \vp_0(2 b_0\Psi_1^0 - 2 a_0 \Psi_2^0-D_2^0) \vp_0'$ is an odd function, so it is perpendicular to $Ker[L_-]=span[\vp_0]$.

   The first equation requires  the solvability condition $\si \vp_0' + 12 \vp_0 (a_0 \Psi_1^0+b_0 \Psi_2^0) \vp_0'\perp \vp_0'$. Noting that $a_0 \Psi_1^0+b_0 \Psi_2^0=a_0 L_+^{-1}[1]$,   we derive the formula for $\si_0$,
   \begin{equation}
   \label{b:50}
   \si_0=-12 \f{\dpr{\vp_0 \vp_0'(a_0 \Psi_1^0+b_0 \Psi_2^0)}{\vp_0'}}{\|\vp_0'\|^2}=
   -12 \f{a_0}{\|\vp_0'\|^2} \dpr{ L_{+}^{-1}[1]}{\vp_0 (\vp_0')^2}.
   \end{equation}
   This quantity can be computed explicitly, in fact  $\dpr{ L_{+}^{-1}[1]}{\vp_0 (\vp_0')^2}=0$, see Proposition \ref{antany} below. Our analysis is not precise enough\footnote{we do not have the precise asymptotic expressions of order  $O(h^2)$ above, although this is in principle  possible, after heavy computations.} to compute explicitly the next order, $O(h^2)$ term.

   All this means is that the translational eigenvalue for $\cl_{+,h}$  is zero to first order in $h$. As we have shown above, with index counting calculations, it turns out that this eigenvalue must have moved to the left (of order $O(h^2)$) to become a negative one.

   \section{Computation of the relevant quantities involvinf $L_+^{-1}$}

   \begin{proposition}
   	\label{antany}
   In the setup of Proposition \ref{prop:h0}, we have the following formulas
   \begin{eqnarray}
   \label{ant1}
   & & \dpr{L_+^{-1} \vp_0}{\vp_0} <0 \\
   \label{ant2}
 & &   \dpr{L_+^{-1}[1]}{\vp_0}=0, \\
 & & \label{ant3}
 \dpr{ L_{+}^{-1}[1]}{\vp_0 (\vp_0')^2}=0 \\
   \label{ant4}    	
   & & 	\int_0^T \f{(2-6\vp_0^2(y))((\vp_0'(y))^2-(\vp_0''(y))^2)}{((\vp_0'(y))^2+(\vp_0''(y))^2))^2} dy\neq 0
\end{eqnarray}
   \end{proposition}
   \begin{proof}
  For the proof of \eqref{ant1},  recall the basic properties at $h=0$. In this case the equation (\ref{1.2}) has a solution
$$\varphi_0(x)=\alpha dn (\alpha x, \kappa), $$
where $\kappa^2=\frac{2\alpha^2-1}{\alpha^2}$. Also $L_+=-\p_x^2+1-2(\vp_0)^2$. Since $L_+\vp_0'=0$,  the function
      $$\psi(x)=\vp_0'(x)\int^{x}{\frac{1}{\vp_0'^2(s)}}ds, \; \; \left| \begin{array}{cc} \vp_0'& \psi \\ \vp_0'' & \psi'\end{array}\right|=1 $$
      is also solution of $L_+\psi=0$.
      Formally, since $\vp_0'$ has   zeros  using the identities
        $$\frac{1}{cn^2(y,\kappa)}=\frac{1}{dn(y, \kappa)}\frac{\partial}{\partial_y}\frac{sn(x, \kappa)}{cn(y, \kappa)}, \; \;
        \frac{1}{sn^2(y,\kappa)}=-\frac{1}{dn(y, \kappa)}\frac{\partial}{\partial_y}\frac{cn(x, \kappa)}{sn(y, \kappa)}$$
        and integrating by parts we get
       $$
        \psi(x) =\frac{1}{\alpha^2\kappa^2}\left[\frac{1-2sn^2(\alpha x, \kappa )}{dn (\alpha x, \kappa) }
        - \alpha \kappa^2sn(\alpha x, \kappa)cn(\alpha x, \kappa)  \int_{0}^{x}{\frac{1-2sn^2(\alpha s, \kappa )}{dn^2(\alpha s, \kappa)}}ds\right]
       $$
       Thus, we may construct Green function
         $$L_+^{-1}f=\vp_0'\int_{0}^{x}{\psi(s)f(s)}ds-\psi(s)\int_{0}^{x}{\vp_0'(s)f(s)}s+C_f\psi(x), $$
         where $C_f$ is chosen such that $L_{+,0}^{-1}f$ is periodic with same period as $\vp_0(x)$.
         After integrating by parts, we get
           \begin{equation}\label{L1}
             \langle L_+^{-1}\vp_0, \vp_0\rangle = -\langle \vp_0^3, \psi \rangle +\frac{\vp_0^2(T)+\vp_0^2}{2}\langle \vp_0, \psi \rangle +C_{\vp_0}\langle \vp_0, \psi \rangle,
           \end{equation}
   We have
             \begin{equation}\label{L2}
             \begin{array}{ll}
               \langle \vp_0, \psi \rangle =\frac{1}{\alpha^2 \kappa^2}[E(\kappa)-K(\kappa)] \\
               \\
               \langle \vp_0^3, \psi \rangle =\frac{1}{2\kappa^2} [(2-\kappa^2)E(\kappa)-2(1-\kappa^2)K(\kappa)] \\
               \\
               C_{\vp_0}=-\frac{\vp_0''(T)}{2\psi(T)}\langle \vp_0, \psi \rangle+\frac{\vp_0^2(T)-\vp_0^2(0)}{2}.
               \end{array}
             \end{equation}
            With this finally we get
             $$\dpr{L_+^{-1} \vp_0} {\vp_0} = \frac{E^2(\kappa)-(1-\kappa^2)K^2(\kappa)}{2[2(1-\kappa^2)K(\kappa)-(2-\kappa^2)E(\kappa)]}<0.
             $$
             see Figure \ref{fig1} below.
             \begin{figure}[h]
             	\centering
             	\includegraphics[width=8cm,height=6cm]{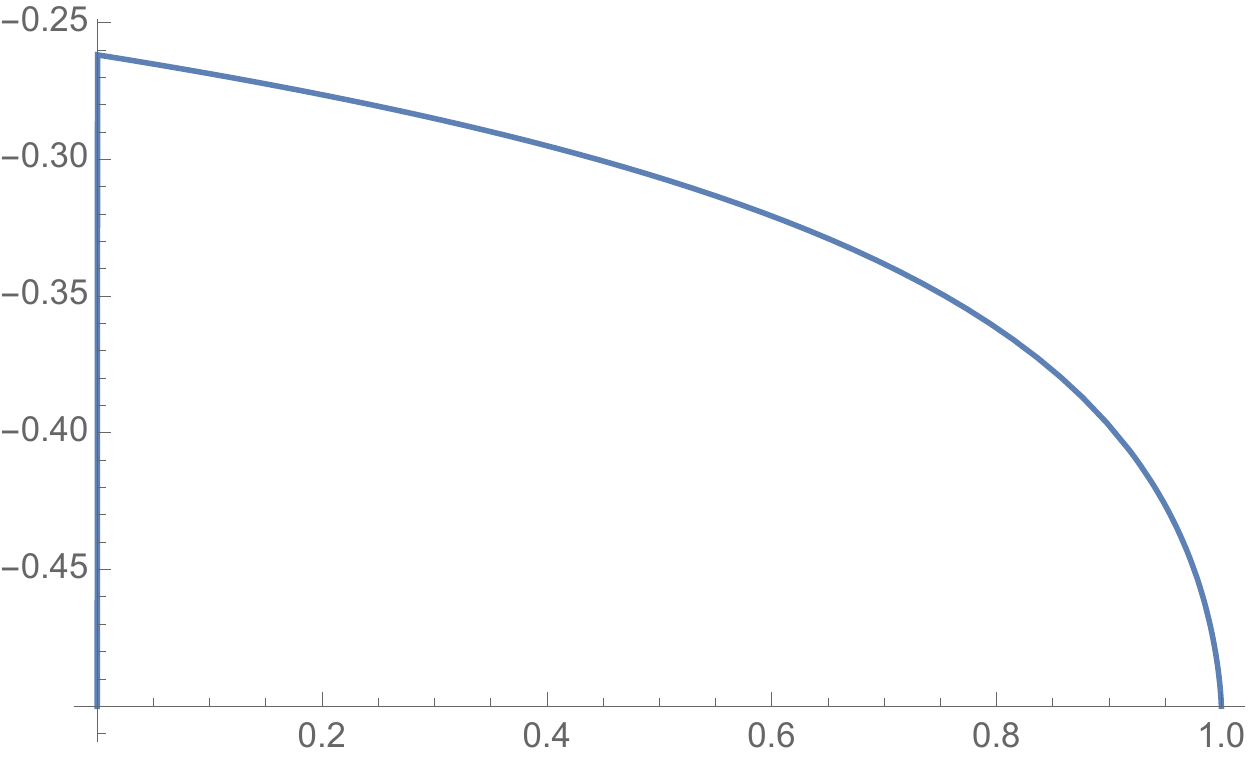}
             	\caption{The function $k\to\frac{E^2(\kappa)-(1-\kappa^2)K^2(\kappa)}{2[2(1-\kappa^2)K(\kappa)-(2-\kappa^2)E(\kappa)]}$ }
             	\label{fig1}
             \end{figure}

              For the proof of \eqref{ant2},  we have
             $$L_+^{-1}[1]=\varphi_0'\int_{0}^{x}\psi (s) ds-\psi (x)\int_{0}^{x}\varphi_0'(s)ds, $$
             and
             $$\dpr{L_+^{-1} [1]} {\vp_0}=\frac{\varphi_0^2(T)}{2}\int_{-T}^{T}{\psi(x)}dx-\frac{3}{2}\int_{-T}^{T}{\psi(x)\varphi_0^2(x)}dx+[C_1+\varphi_0(0)]\int_{-T}^{T}{\psi(x)\varphi_0(x)}dx,$$
             where$$C_1=\varphi_0(T)-\varphi (0)-\frac{\varphi_0''(T)}{\psi'(T)}\int_{0}^{T}{\psi(x)}dx.$$
             Using that $\frac{d}{dx} dn(x)=-\kappa^2sn(x)cn(x)$ and integrating be parts, we get  
             $$\begin{array}{ll}
             \int_{-T}^{T}{\psi(x)\varphi_0(x)}dx & =\frac{2}{\kappa^2} \left[ \frac{2}{3}\int_{0}^{T}{dn(\alpha x)(1-2sn^2(\alpha x))}dx+\frac{1}{3}dn^3(\alpha T)\int_{0}^{T}{\frac{1-2sn^2(\alpha x)}{dn^2(\alpha x)}}dx\right] \\
             \\
              & =\frac{2}{3\kappa^2}dn^3(K)\int_{0}^{T}{\frac{1-2sn^2(\alpha x)}{dn^2(\alpha x)}}dx.
             \end{array} $$
             Similarly, integrating be parts
               $$\int_{-T}^{T}{\psi(x)}dx=\frac{2dn(\alpha T)}{\alpha^2\kappa^2}\int_{0}^{T}{\frac{1-2sn^2(\alpha x)}{dn^2(\alpha x)}}dx.$$
               Thus,
             $$
             \frac{\varphi_0^2(T)}{2}\int_{-T}^{T}{\psi(x)}dx-\frac{3}{2}\int_{-T}^{T}{\psi(x)\varphi_0^2(x)}dx=0.
             $$
             Using that $\varphi_0''(T)=\alpha^3\kappa^2\sqrt{1-\kappa^2}$, and $\psi'(T)=\sqrt{1-\kappa^2}\int_{0}^{T}{\frac{1-2sn^2(\alpha x)}{dn^2(\alpha x)}}dx$, we get
             $$
             C_1+\varphi_0(0)=\varphi_0(T)-\frac{\varphi_0''(T)}{\psi'(T)}\int_{0}^{T}{\psi(x)}=0.
             $$
             For the proof of \eqref{ant3},
            \begin{eqnarray*}
      \dpr{L_+^{-1} [1]} {\vp_0(x)(\varphi_0'(x))^2} &=& \int_{-T}^{T}{\varphi_0(x)(\varphi_0'(x))^3\int_{0}^{x}{\psi(s)}ds}dx-
      \int_{-T}^{T}{\psi(x)\varphi_0^2(x)(\varphi_0'(x))^2}dx\\
        	& + & (C_1+\varphi_0(0))\int_{-T}^{T}{\psi(x)\varphi_0(x)(\varphi_0'(x))^2}dx.
            \end{eqnarray*}

             Again integrating by parts, we get 
             $$\int_{0}^{x}{\psi(s)}ds=\frac{dn(\alpha x)}{\alpha^2 \kappa^2}\int_{0}^{x}{\frac{1-2sn^2(\alpha s)}{dn^2(\alpha s)}}ds,$$
              and
             $$
             \begin{array}{ll}
             \int_{-T}^{T}{\varphi_0(x)(\varphi_0'(x))^3\int_{0}^{x}{\psi(s)}ds}dx-\int_{-T}^{T}{\psi(x)\varphi_0^2(x)(\varphi_0'(x))^2}dx\\
             \\
             =\int_{-T}^{T}{\varphi_0(x)(\varphi_0'(x))^2\left[ \varphi_0'(x)\int_{0}^{x}{\psi(s)}ds-\psi(x)\varphi_0(x)\right]}dx \\
             \\
             =-\alpha^3\kappa^2\int_{0}^{K(\kappa)}{(1-2sn^2(x))sn^2(x)cn^2(x)dn(x)}dx=0.
             \end{array} $$

 Now, using that $\kappa^2=\frac{2\alpha^2-1}{\alpha^2}$, we get
               $$
               \begin{array}{ll}
               \int_0^T \f{(2-6\vp_0^2(y))((\vp_0'(y))^2-(\vp_0''(y))^2)}{((\vp_0'(y))^2+(\vp_0''(y))^2))^2} dy
               =\int_{0}^{K(\kappa)}{\frac{\left(2-\frac{6}{2-\kappa^2}dn^2(y)\right)\left( sn^2(y)cn^2(y)-\frac{1}{2-\kappa^2}[cn^2(y)-sn^2(y)]^2dn^2(y)\right)}{\left[ sn^2(y)cn^2(y)+\frac{1}{2-\kappa^2}[cn^2(y)-sn^2(y)]^2dn^2(y)\right]^2}}dy.
               \end{array}
               $$
               Using Mathematica, we were able to compute this last expression
               $$
               \int_0^T \f{(2-6\vp_0^2(y))((\vp_0'(y))^2-(\vp_0''(y))^2)}{((\vp_0'(y))^2+(\vp_0''(y))^2))^2} dy = 2 K(k) - \f{(2-k^2)}{1-k^2} E(k)<0,
               $$
               as is clear from Figure \ref{fig2} below.
               \begin{figure}[hh]
               	\centering
               	\includegraphics[width=8cm,height=6cm]{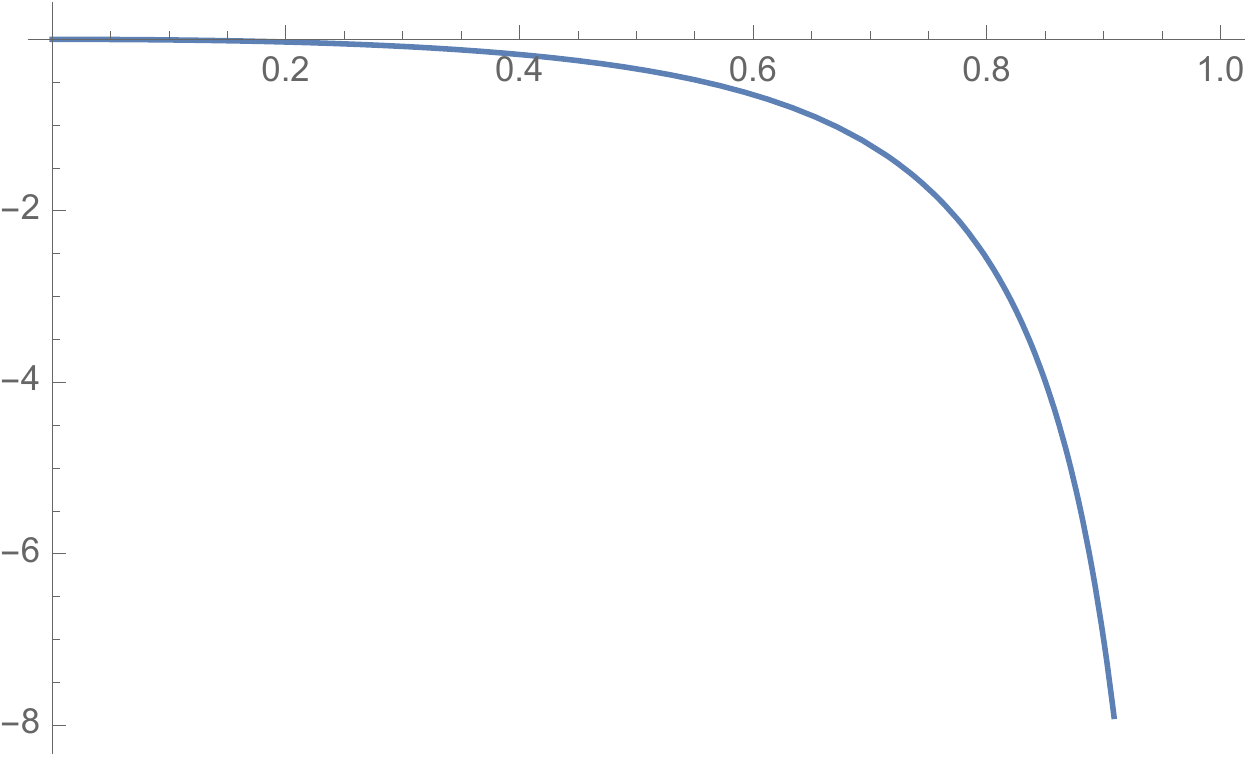}
               	\caption{The function $k \to 2 K(k) - \f{(2-k^2)}{1-k^2} E(k).$ }
              \label{fig2}
              \end{figure}
   \end{proof}

   %   \hline

% \begin{tikzpicture}
% \begin{axis}[
% axis lines = left,
%  ylabel = { },
% ]
% %Below the red parabola is defined
% \addplot [
% domain=-10:10,
% samples=100,
% color=red,
% ]
% %{y=-1}
% {x^2 - 2*x - 1};
% \addlegendentry{$x^2 - 2x - 1$}
% %Here the blue parabloa is defined
% \addplot [
% domain=-10:10,
% samples=100,
% color=blue,
% ]
% {x^2 + 2*x + 1};
% \addlegendentry{$x^2 + 2x + 1$}
%
% \end{axis}
% \end{tikzpicture}

 %\newpage

\end{document}